\documentclass[journal, 12pt, onecolumn]{IEEEtran}

\usepackage{cite}
\usepackage{graphicx}
\usepackage{subfigure}
\usepackage{url}
\usepackage{stfloats}
\usepackage{latexsym}
\usepackage{amsfonts}
\usepackage{amsmath}
\usepackage{array}
\usepackage{textcomp}
\usepackage{pxfonts}
\usepackage{txfonts}
\usepackage{cases}
\usepackage{psfrag}
\usepackage{epsfig}
 \usepackage{setspace}
\newtheorem{definition}{Definition}

\newtheorem{example}{Example}
\newtheorem{solution}{Solution}
\newtheorem{theorem}{Theorem}

\begin{document}

\title{\Large  On stability analysis by using Nyquist and Nichols Charts}

\author{S.M.Mahdi Alavi and Mehrdad Saif
\thanks{S.M.M. Alavi is with the Brain Stimulation Engineering Laboratory, Duke University,  Durham, NC 27710, USA. Email to: {\tt\small mahdi.alavi@duke.edu}.}
\thanks{M. Saif is with the Department of Electrical and Computer Engineering, University of Windsor, Windsor, ON N9B 3P4, Canada. Email to: {\tt\small msaif@uwindsor.ca}}
\thanks{This work was accomplished during the appointment of the first author at the University of Windsor between 2012 and 2013. The authors would like to acknowledge the partial support from Natural Sciences and Engineering Research Council of Canada (NSERC).}}

\markboth{ }%
{Shell \MakeLowercase{\textit{et al.}}: Bare Demo of IEEEtran.cls
for Journals}
\maketitle

\begin{abstract}
This paper reviews stability analysis techniques by using the Nyquist and Nichols charts. The relationship between the Nyquist and Nichols stability criteria is fully described by using the crossing concept. The results are demonstrated through several numerical examples. This tutorial provides useful insights into the loop-shaping based control systems design such as Quantitative Feedback Theory.  
\end{abstract}

\begin{IEEEkeywords}
Stability analysis, Nyquist diagram, Nichols chart, Quantitative feedback theroy (QFT), loop-shaping.
\end{IEEEkeywords}

\IEEEpeerreviewmaketitle

\section{Introduction}
Stability is always a major concern in analysis and design of feedback control systems. Consider the Linear Time Invariant (LTI) feedback system as shown in Fig. \ref{Fig:1DOF_standard_Feedback_structure}. $P(s)$ and $G(s)$ represent plant and controller Transfer Functions (TFs).
\begin{figure}[htb]
\psfrag{PPPP\r}[c][c]{$P(s)$} \psfrag{PC\r}[c][c]{$G(s)$}
\psfrag{y-t-\r}{$y(t)$} \psfrag{r-t-\r}{$r(t)$}
    \centering
   \includegraphics[scale=1]{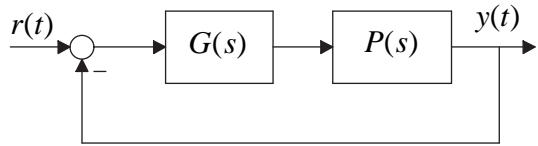}
     \caption{The standard feedback system.}
     \label{Fig:1DOF_standard_Feedback_structure}
\end{figure}

For the stability of the feedback system Fig. \ref{Fig:1DOF_standard_Feedback_structure}, the poles of the closed-loop system must have negative real parts and lie in the Left Half Plane, (LHP). The analogous is that the zeros of the characteristic function, $Q(s)$, given by
\begin{align}
\label{system_characteristics_function} Q(s)=1+L(s), \end{align}
must be located in LHP to ensure the stability of the feedback system Fig.  \ref{Fig:1DOF_standard_Feedback_structure}, where $L(s)=P(s)G(s)$ refers to the `loop function' in the control literature.

A variety of stability criteria are available in the literature, \cite{Houpis-LinearCont}. The Routh-Hurwitz and Jury tests determine the stability of the LTI systems without finding the roots of the characteristic equation in continuous and discrete time domains, respectively. Nyquist diagram is an alternative stability analysis method, providing some graphical information on how the poles of the closed-loop control system move by changing the gain of the controller in the complex plane.

Nichols chart is an alternative coordinates for presentation of the system's frequency response. In the Nichols chart, the loop-function's magnitude (in dB) versus its phase (in Degrees) is plotted. Since the Nichols chart presents a complete simultaneous information in relation to magnitude and phase of the system frequency response, a control approach that employs it as the design environment is more efficient than other control approaches that are only based on the magnitude or phase information. The use of both magnitude and phase information into the design of the feedback system significantly eases stability analysis of non-minimum phase, unstable plants and the plants that suffer from time-delay.

Despite the Nyquist stability has very well been established in the linear control text books, there is still significant demands for more discussions and examples to illustrate the Nichols chart stability criterion. This paper summarizes the results in \cite{Cohen-1994} and \cite{Chen-Ballance1997} in a tutorial fashion. The relationship between the Nyquist and Nichols stability criteria is exactly described through several illustrative examples.

As an application, the Quantitative Feedback Theory (QFT) \cite{Houpis-QFT} employs the Nichols chart as a tool \cite{Borhgesani-QFTToolbox} and \cite{Garcia-Sanz-QFTToolbox} for the design of the feedback compensator. In this approach, desired control objectives are given in terms of some graphical design bounds on the system's loop-function. The feedback controller is then obtained by loop-function shaping such that these design bounds are satisfied. Some applications of QFT include: fault diagnosis and control in chemical processes \cite{Alavi-2012, Alavi-2007} and in power systems \cite{Alavi-2008}, power control in wireless sensor networks \cite{Alavi-2009a}, active queue management in communications systems \cite{Alavi-2009b} to name but a few. From quantitative feedback design perspective, this tutorial provides useful insights into the loop-shaping process. 

The rest of the paper is organized as follows. In section \ref{sec:Stability Analysis Using Nyquist Plot}, the Nyquist stability test is described. In section \ref{sec:Stability Analysis Using Nichols Chart}, the stability criterion by using the Nichols chart, and its relation to the Nyquist stability test are presented. Several examples are given in section \ref{sec:Examples}, demonstrating the results.

\section{Stability Analysis Using Nyquist Plot}
\label{sec:Stability Analysis Using Nyquist Plot}
The Nyquist stability criterion is based on the Cauchy's principle. For the feedback system Fig. \ref{Fig:1DOF_standard_Feedback_structure} with the loop-function $L(s)$, consider a contour in s-plane enclosing the entire Right-Half Plane (RHP) including the imaginary axis. This contour refers to the `standard Nyquist contour' and is denoted by $\Gamma$ hereafter. The requirement for the Cauchy's principle is that $\Gamma$ must not pass through any pole of $L(s)$. Its radius should be large enough to enclose all Right-Half Plane (RHF) poles of $L(s)$. To avoid passing through the poles located on the imaginary axis, $\Gamma$ is constructed such that it encircles these poles by a semicircular arc with a very small radius tending to zero. Fig. \ref{Fig:standard_Nyquist_contour} shows how the standard Nyquist contour is chosen in the presence of different locations of poles. Zeros of $L(s)$ do not affect the selection of $\Gamma$. The Nyquist plot is then the map of $\Gamma$ trough $L(s)$ into the complex plane. For the sake of simplicity, it is assumed that there is no pole-zero cancellation between the numerator and denominator of $L(s)$.

\begin{figure}[htb]
\psfrag{Re\r}[c][c]{\scriptsize $Re\{s\}$} \psfrag{Im\r}{\scriptsize $Im\{s\}$}
\psfrag{g\r}[c][c]{$\Gamma$} \psfrag{splane\r}[c][c]{\scriptsize $s$-plane}
    \centering
   \includegraphics[scale=1]{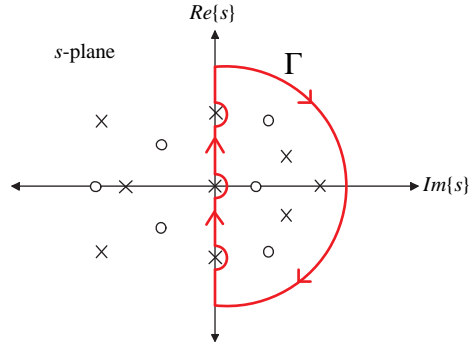}
     \caption{The standard Nyquist contour $\Gamma$. `$\times$'s and `$\circ$'s represent
      the possible poles and zeros of $L(s)$, respectively. $\Gamma$ must be large enough to enclose all RHP poles of $L(s)$. The poles located on the imaginary axis must be detoured by $\Gamma$.}
     \label{Fig:standard_Nyquist_contour}
\end{figure}

By using Cauchy's principle, the following equation holds for the obtained Nyquist plot of $L(s)$ in the complex plane,
\begin{align}
\label{Cauchy_principle_Eq} N_z=N+N_p,
\end{align}
where, $N_z$ and $N_p$ are the number of zeros and poles of $L(s)$ inside the contour $\Gamma$. $N$ is the number of times that the Nyquist plot encircles the origin.

According to (\ref{Cauchy_principle_Eq}), the number of times that Nyquist plot of $L(s)$ encircles the critical point $(-1,0)$ plus the number of poles of $L(s)$ determine the number of zeros of $Q(s)$ which are located inside the contour $\Gamma$. From equation (\ref{system_characteristics_function}), the zeros of $Q(s)$ which are located inside the contour $\Gamma$ are unstable poles of the feedback system Fig. \ref{Fig:1DOF_standard_Feedback_structure}. Therefore, the number of times that Nyquist plot of $L(s)$ encircles the critical point $(-1,0)$ plus the number of poles of $L(s)$ determine the number of unstable poles of the closed-loop system through the following theorem.

\begin{theorem}
\label{Theorem:Nyquist_stability_Criterion} Consider the feedback system as shown in Fig. \ref{Fig:1DOF_standard_Feedback_structure}. The followings are then equivalent:
\begin{itemize}
\item[-] The number of unstable poles of the feedback system Fig. \ref{Fig:1DOF_standard_Feedback_structure}, $N_z$, is given by:
\begin{align}
\label{Nyquist_Criterion_Eq} N_z=N+N_p,
\end{align}
$N_p$ is the number of poles of $L(s)$ inside the contour $\Gamma$ and $N$ is the number of times that the Nyquist plot encircles the critical point $(-1,0)$.
\item[-] The feedback system is stable if and only if the Nyquist diagram of $L(s)$ does not intersect the
critical point $(-1,0)$ and encircles it `$n$' times in the counterclockwise direction.\end{itemize}
\end{theorem}
\begin{proof}
This is a direct consequence of the Cauchy's principle. For a more comprehensive information interested readers are directed to \cite{Houpis-LinearCont}.
\end{proof}

It should be noted that an encirclement of the point $(-1,0)$ is counted as positive if it is in the same direction as the standard Nyquist contour and negative if it is in the opposite direction. In this context, since the standard Nyquist contour is clockwise, if $L(s)$ generates clockwise encirclement of the point $(-1,0)$, then $N$ is added with $+1$. Otherwise if $L(s)$ encircles the point $(-1,0)$ in counterclockwise, then $N$ is added with $-1$.

In the following, an approach is presented that simplifies counting $N$ in (\ref{Nyquist_Criterion_Eq}). It provides the basis for development of the stability criterion using Nichols chart.

\begin{definition}\label{Def:Smooth_curve} A Nyquist diagram is a smooth curve if it is differentiable for all $\theta\in [0,2\pi]$ where, $\theta$ represents the phase of the loop-function at each frequency, and is given by:
\begin{equation}
\label{phase_in_complex_plane_Eq}
\theta=\tan^{-1}(Im\{L(j\omega)\}/Re\{L(j\omega)\}).
\end{equation}
\end{definition}

\begin{definition}\label{Def:Crossing} Let $\mathcal{R}_{(-\infty,-1)}$ be the ray on the real axis from $-\infty$ to $-1$ in the complex plane. A crossing occurs when $L(s)$ intersects $\mathcal{R}_{(-\infty,-1)}$. The crossing is said to be positive if the tangent to the Nyquist plot
has a positive imaginary part. Consequently, it is negative if the tangent to the Nyquist plot has a negative imaginary part, as shown in Fig. \ref{Fig:crossing_Nyquist_plot}.\end{definition}

\begin{figure}[htb]
\psfrag{Re\r}[c][c]{\scriptsize $Re\{L(s)\}$} 
\psfrag{Im\r}{\scriptsize $Im\{L(s)\}$}
\psfrag{CP\r}[c][c]{\scriptsize $-1$}
\psfrag{R\r}[c][c]{\scriptsize $\mathcal{R}_{(-\infty,-1)}$}
\psfrag{Lplane\r}[c][c]{\scriptsize $L$-plane}
    \centering
   \includegraphics[scale=1]{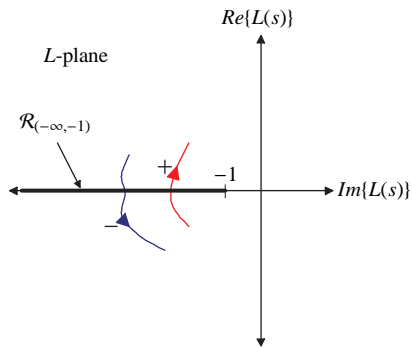}
     \caption{The crossing notation of a Nyquist diagram $L(s)$.}
     \label{Fig:crossing_Nyquist_plot}
\end{figure}

If the Nyquist plot is a smooth curve, the sign of crossing is obviously $+1$ or $-1$. However, when the Nyquist plot approaching a point on the real axis, there might be a {\em cusp} due to local singularity of the Nyquist plot. The crossing is not {\em well-defined} as the Nyquist plot is not differentiable at such a point. This situation might occur at $\omega=0$ or $\omega =\pm \infty$. To address this issue when there exists a cusp at $\omega=0$, the sign of the first nonzero derivative of $\theta$, given by (\ref{phase_in_complex_plane_Eq}), with respect to $\omega$ at $\omega=0$ is adopted as the sign of the crossing. If this derivative is positive, the crossing is supposed to be positive, otherwise, it is selected as negative. For the cusp at $\omega=\pm \infty$, because all derivatives of $\theta$ approach to zero as $\omega \rightarrow \pm \infty$, the above solution is not feasible. Instead, $\omega$ is replaced with $1/\omega$ and the derivation technique is again applied but at this time with respect to $1/\omega$.

It is then a direct consequence of the above statements that: `A crossing is said to be positive if the direction of the Nyquist plot is upward, otherwise the crossing is negative.'

\begin{theorem}
\label{Theorem:Modified_Nyquist_stability_Criterion} Consider the feedback system as shown in Fig. \ref{Fig:1DOF_standard_Feedback_structure}. Assume that $L(s)$ has `$n \geq 0$' unstable poles. By using the crossing concepts as discussed above, the followings are then equivalent.
\begin{itemize} \item[-] The sum of crossings is equal to the number of encirclements of the critical point $(-1,0)$ by the Nyquist diagram of $L(s)$.
\item[-] The feedback system is stable if and only if the Nyquist diagram of $L(s)$ does not intersect the critical point $(-1,0)$ and sum of its crossings is equal to `$-n$'.\end{itemize}
\end{theorem}
\begin{proof}
This is an immediate consequence of the crossing concept and Theorem \ref{Theorem:Nyquist_stability_Criterion}.
\end{proof}

The number of RHP zeros of $Q(s)$ is obtained by substituting the sum of crossings as well as the number of poles of $L(s)$ into (\ref{Nyquist_Criterion_Eq}).

\section{Stability Analysis Using Nichols Chart}
\label{sec:Stability Analysis Using Nichols Chart}

In the following, the stability criterion of the feedback system Fig. \ref{Fig:1DOF_standard_Feedback_structure} is studied which is based on the Nichols chart \cite{Nichols-Theory-of-Servomechanisms}. Again, it is assumed that no unstable pole/zero cancellation takes place in $L(s)$. The stability criterion is obtained by mapping the critical point $(-1,0)$, the ray $\mathcal{R}_{(-\infty,-1)}$ and the Nyquist plot of $L(s)$ into the Nichols chart, and by using the crossing concept.

$\Upsilon$ is the map that transforms any point $A(x,y)$ on the Nyquist diagram of $L(s)$ into its equivalent point, $A^\prime(\phi,r)$, in the Nichols chart, through:
\begin{equation}
\label{Map_from_Nyquist_To_NC_Eq} \Upsilon: A(x,y) \rightarrow A^\prime(\phi,r)
\end{equation}
\begin{align} \label{Phase_in_NC_inTerms_xy}& \phi=\tan^{-1}(y/x),\\
\label{Mag_in_NC_inTerms_xy} & r=20\log(\sqrt{x^2+y^2}).\end{align}
$x$ and $y$ denote the real and imaginary parts of the Nyquist plot of $L(s)$ in the complex plane.

\subsection{Stability Analysis Using Single-Sheeted Nichols Chart}
There are two aspects in the presentation of the system's frequency response in the Nichols chart. In the first aspect, the exact phase obtained through $\Upsilon$ is plotted inside the Nichols chart. The Nichols chart that covers phase information outside the range of $[-360^\circ,0)$, is so-called as the multiple-sheeted Nichols chart. Subsequently, the Nichols curve that is plotted in multiple-sheeted Nichols chart refers to multiple-sheeted Nichols plot.

The second aspect relies on this fact that the system's characteristics do not change if the phase response is shifted by $(\pm360k)^\circ$ for $k=0,1,\cdots$. Therefore, any segment of the multiple-sheeted Nichols plot outside the range of $[-360^\circ,0)$ can be horizontally shifted by integer multiples of $\pm360^\circ$ to be located inside $[-360^\circ,0)$. The Nichols chart with horizontal phase axis within the range of $[-360^\circ,0)$ is so-called as the single-sheeted Nichols chart. Subsequently, the Nichols curve that is plotted in the single-sheeted Nichols chart refers to single-sheeted Nichols plot.

In this section, the stability criterion of the feedback system Fig. \ref{Fig:1DOF_standard_Feedback_structure} using single-sheeted Nichols chart, is derived as follows.

Let $\mathcal{R}_{(-180^\circ,r>0dB)}$ be the ray of all points on the line $r>0dB$ and $\phi=-180^\circ$ in the single-sheeted Nichols chart. By simple calculations, the critical point $(-1,0)$ and the ray $\mathcal{R}_{(-\infty,-1)}$ in complex plane are respectively mapped through $\Upsilon$ to the point $(-180^\circ,0dB)$ and the ray $\mathcal{R}_{(-180^\circ,r>0dB)}$ in the single-sheeted Nichols chart as shown in Fig. \ref{Fig:crossing_Nichols_chart}.

\begin{figure}[htb]
\psfrag{Re\r}[c][c]{\scriptsize $20\log r$dB}
\psfrag{Im\r}{\scriptsize $\phi$(degrees)}
\psfrag{C0\r}[c][c]{\scriptsize $0^\circ$}
\psfrag{C1\r}[c][c]{\scriptsize $-180^\circ$}
\psfrag{C2\r}[c][c]{\scriptsize $-360^\circ$}
\psfrag{db\r}[c][c]{\scriptsize $0$dB}
\psfrag{R\r}{\scriptsize $\mathcal{R}_{(-180^\circ,r>0dB)}$}
\psfrag{NC\r}[c][c]{\scriptsize Nichols chart}
    \centering
   \includegraphics[scale=1]{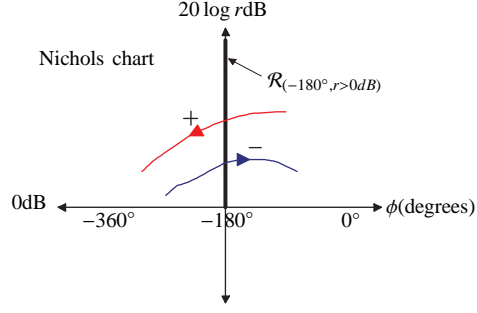}
     \caption{Illustration of the ray $\mathcal{R}_{(-180^\circ,r>0dB)}$ and crossing signs in the single-sheeted Nichols chart.}
     \label{Fig:crossing_Nichols_chart}
\end{figure}

Consider the Nyquist plot of a given $L(s)$ which is a curve denoted by $\Psi_{ny}$ in the complex plane. As mentioned earlier, $\Psi_{ny}$ is a continues and closed curve, symmetric with respect to the real axis. When mapping $\Psi_{ny}$ through $\Upsilon$ into the single-sheeted Nichols chart, the resulting curve, which is denoted by $\Psi_{nc}$, intersects the ray $\mathcal{R}_{(-180^\circ,r>0dB)}$ at the points that are maps of the crossings of $\Psi_{ny}$ in the complex plane.

For a positive crossing of $\Psi_{ny}$ in the complex plane, $\Psi_{nc}$ will pass $\mathcal{R}_{(-180^\circ,r>0dB)}$ to the left side in the single-sheeted Nichols chart. Similarly, For a negative crossing of $\Psi_{ny}$ in the complex plane, $\Psi_{nc}$ will pass $\mathcal{R}_{(-180^\circ,r>0dB)}$ to the right side in the single-sheeted Nichols chart.

For a given Nyquist diagram of $L(s)$ and its types of crossings in complex plane as shown in Fig. \ref{Fig:crossing_Nyquist_plot}, Fig. \ref{Fig:crossing_Nichols_chart} demonstrates the corresponding Nichols curve and types of the equivalent crossings in the single-sheeted Nichols chart.

\begin{definition}\label{Def:Crossing_in_NC}Let
$\mathcal{R}_{(r>0dB,-180^\circ)}$ be the ray on the line $r>0dB$ and $\phi=-180^\circ$ in the single-sheeted Nichols chart. A crossing occurs when the single-sheeted Nichols plot of $L(s)$ intersects $\mathcal{R}_{(r>0dB,-180^\circ)}$. For consistency with the Definition \ref{Def:Crossing}, the crossing is said to be positive if the direction of the single-sheeted Nichols plot of $L(s)$ is toward left of $\mathcal{R}_{(r>0dB,-180^\circ)}$, otherwise it is negative.
\end{definition}

\begin{theorem}
\label{Theorem:Stability_in_Single_Sheeted_NC} Consider the feedback system as shown in Fig. \ref{Fig:1DOF_standard_Feedback_structure}. Assume that $L(s)$ has `$n \geq 0$' unstable poles. By using the crossing concepts as discussed above, the followings are equivalent.
\begin{itemize}
\item[-] The sum of crossings in the single-sheeted Nichols chart demonstrates the number of times that the Nyquist plot of $L(s)$ encircles the critical point $(-1,0)$ in the Nyquist diagram.
\item[-] The feedback system is stable if and only if the single-sheeted Nichols plot of $L(s)$ does not intersect the critical point $(-180^\circ,0dB)$ and sum of its crossing with $\mathcal{R}_{(-180^\circ,r>0dB)}$ is equal to `$-n$'.\end{itemize}
\end{theorem}
\begin{proof}
This is a direct consequence of the above discussion.
\end{proof}

\subsection{Stability Analysis Using Multiple-Sheeted Nichols Chart}
In this section, the proposed stability criterion given by Theorem \ref{Theorem:Stability_in_Single_Sheeted_NC} is extended to the
multiple-sheeted Nichols chart, where the phase of the system frequency response exceeds the range of $[-360^\circ,0)$.

In this case, the map of the critical point $(-1,0)$ through $\Upsilon$, from the complex plane into the multiple-sheeted Nichols chart, is repeated at the points $((-180\pm 360k)^\circ,0dB)$, $k=0,1, \cdots$. In addition, the map of the ray $\mathcal{R}_{(-\infty,-1)}$ through $\Upsilon$, from the complex plane into the multiple-sheeted Nichols chart will result in multiple rays on lines $((-180\pm 360k)^\circ,r>0dB)$, $k=0,1, \cdots$.

By comparing the single and multiple sheeted Nichols plots of $L(s)$, it is seen that the sum of crossings will be constant for both plots. There is only one difference that the Nichols plot of one is horizontally shifted with respect to another. Therefore, the following stability criterion is derived when using the multiple-sheeted Nichols chart.

\begin{theorem}
\label{Theorem:Stability_in_Multiple_Sheeted_NC} Consider the feedback system as shown in Fig.
\ref{Fig:1DOF_standard_Feedback_structure}. Assume that $L(s)$ has `$n \geq 0$' unstable poles. The followings are equivalent when
using the multiple-sheeted Nichols chart.
\begin{itemize}
\item[-] The sum of crossings illustrates the number of times that the Nyquist diagram of $L(s)$ encircles the critical point (-1,0) in the complex plane.
\item[-] The feedback system is stable if and only if the multiple-sheeted Nichols plot of $L(s)$ does not intersect the critical points $((-180 \pm 360k)^\circ,0dB)$ and sum of its crossing with the rays located on the lines $((-180\pm 360k)^\circ,r>0dB)$ is equal to `$-n$', $k=0,1, \cdots$.
\end{itemize}
\end{theorem}
\begin{proof}
This is a direct consequence of the above discussion and extension of Theorem \ref{Theorem:Stability_in_Single_Sheeted_NC} to the
multiple-sheeted Nichols chart.
\end{proof}

\subsubsection{Multiple-Sheeted Nichols Chart versus Single-Sheeted Nichols Chart}
By comparing the single and multiple Nichols charts, the following features are highlighted.

\begin{itemize}
\item[-] There is no priority between multiple-sheeted and single-sheeted Nichols charts or plots. Only in some cases, the plot of the loop-function's frequency response into the multiple-sheeted Nichols will result in the closed curve. However, multiple crossing ray and multiple critical points appear that must be taken into account when analyzing and designing the control system using multiple-sheeted Nichols chart. 
\item[-] What is very important is that the system characteristics will be preserved by shifting the phase response by $\pm(360k)^\circ$ for $k=0,1,\cdots$. Thus, the use of multiple or single sheeted Nichols chart is a matter of convenience only and do not affect the analysis and design of the feedback systems.
\item[-] In general, there exist a vertical line, $\phi=k$, in which both multiple-sheeted and single-sheeted Nichols plots are symmetric to that. Half of the Nichols chart, located in one side of $\phi=k$, is related to the system frequency response for positive $\omega$, and the other half part, located in the other side of $\phi=k$, is related to the system frequency response for negative $\omega$. Since both full and half multiple-sheeted or single-sheeted Nichols charts result in same number of crossings, half of the multiple-sheeted or single-sheeted Nichols is enough for the stability analysis.
\end{itemize}
\section{Examples}
\label{sec:Examples}
In the following several examples are provided that cover a large variety of the loop-function models, stable and unstable zeros and poles.

\begin{figure}
\centering
\psfrag{Im}[c][c]{\scriptsize  $Im\{L(s)\}$}
\psfrag{Re}{\scriptsize  $Re\{L(s)\}$}
\psfrag{w0}{\scriptsize$ \omega=0$}
\psfrag{win}{\scriptsize $\omega=\pm \infty$}
\psfrag{xin}{\scriptsize  $0$}
\psfrag{w0p}[r][r]{\scriptsize$\omega=0^{+}$}
\psfrag{w0m}{\scriptsize $\omega=0^{-}$}
\psfrag{wip}{\scriptsize $\omega=+\infty$}
\psfrag{wim}{\scriptsize $\omega=-\infty$}
\subfigure[K=5]{\label{SubFig:Stable_MP_Ny_Ex1}
\psfrag{R}{\scriptsize  $\mathcal{R}_{(-\infty,-1)}$}
\psfrag{wc}[c][c]{\scriptsize $\omega=\pm3.32$}
\psfrag{x0}{\scriptsize  $5$}
\psfrag{xc}[c][c]{\scriptsize  -0.50}
\includegraphics[scale=.5]{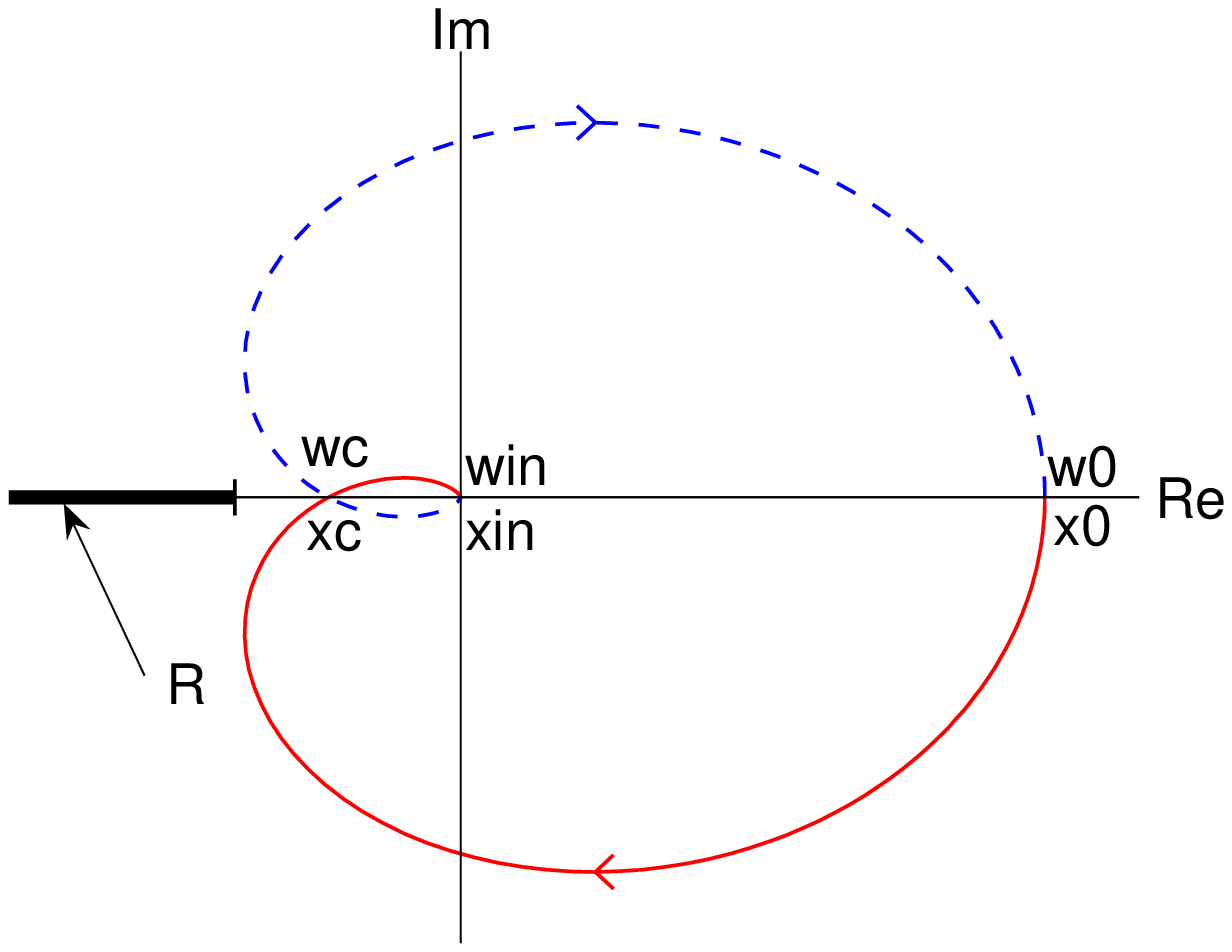}}
\hspace{1em}
\subfigure[K=5]{\label{SubFig:Stable_MP_NC_Ex1}
\psfrag{R}{\scriptsize  $\mathcal{R}_{(-180^\circ,r>0dB)}$}
\psfrag{wc}{\scriptsize $\omega=1.16$}
\psfrag{-80}{\scriptsize  $-\infty$}
\includegraphics[scale=.5]{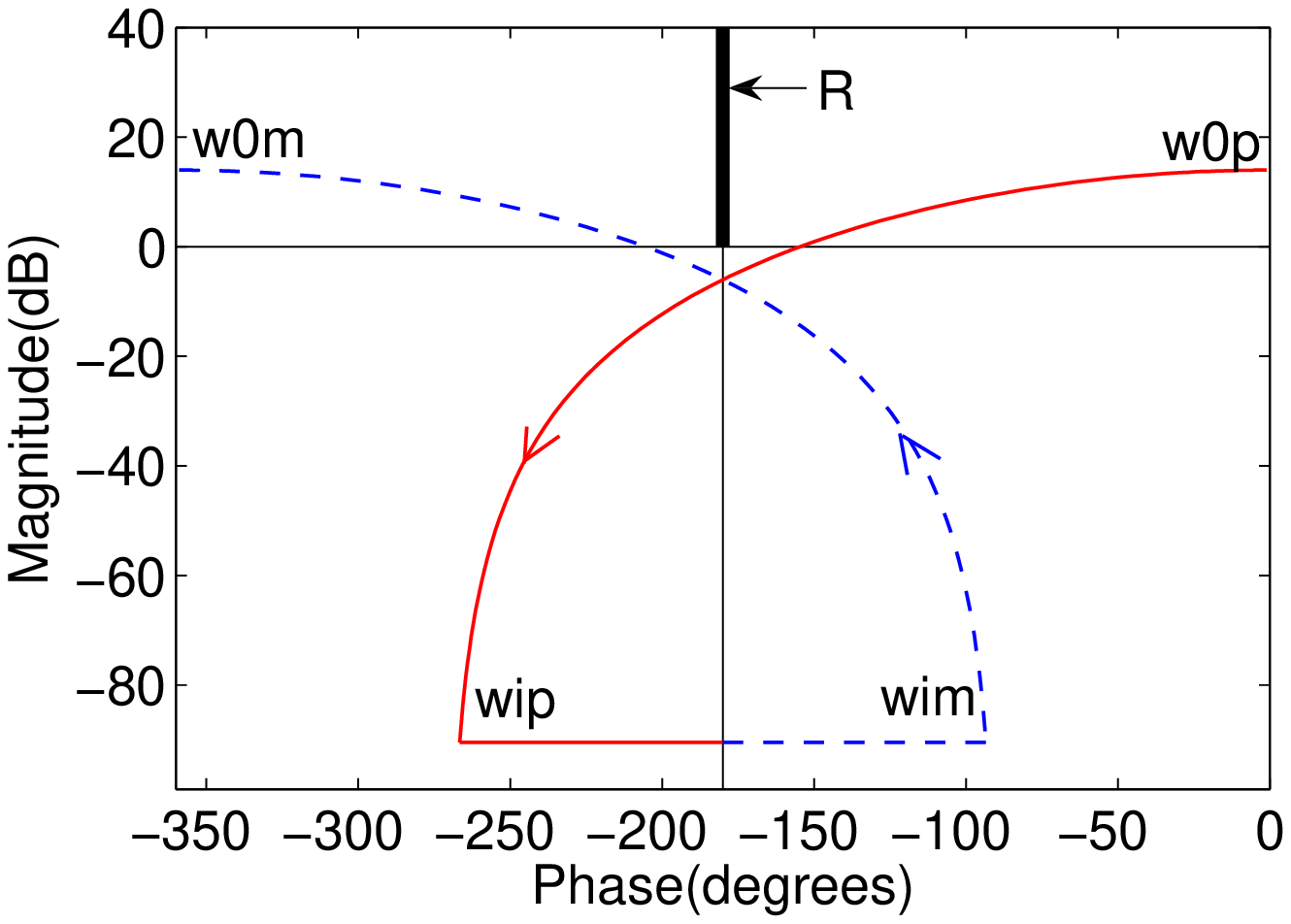}}\\
\subfigure[K=15]{\label{SubFig:Stable_MP_Ny_Ex2}
\psfrag{wc}[c][c]{\scriptsize $\omega=\pm3.78$}
\psfrag{x0}{\scriptsize $15$}
\psfrag{xc}[c][c]{\scriptsize -1.47}
\psfrag{R}{\scriptsize $\mathcal{R}_{(-\infty,-1)}$}
\includegraphics[scale=.5]{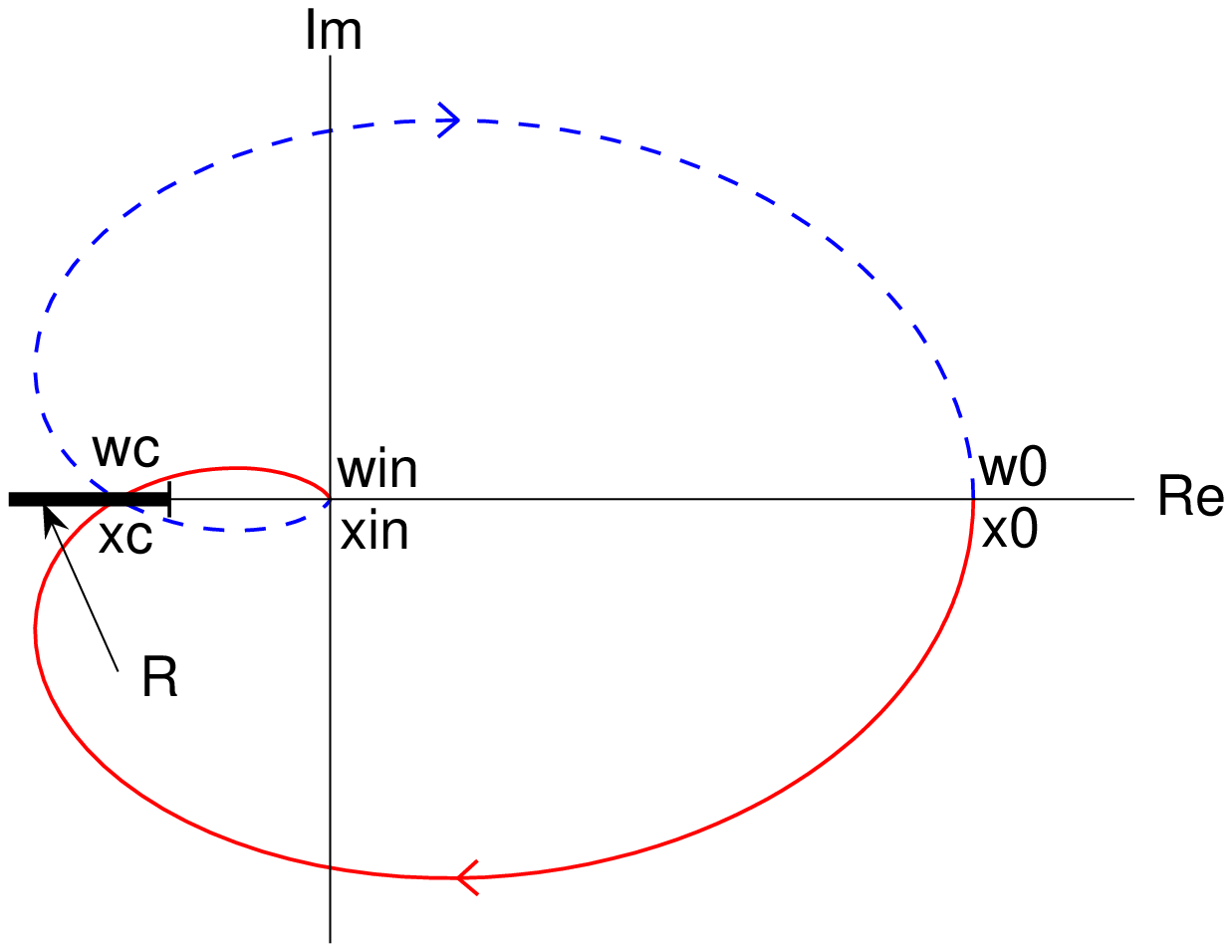}}
\hspace{1em}
\subfigure[K=15]{\label{SubFig:Stable_MP_NC_Ex2}
\psfrag{R}{\scriptsize  $\mathcal{R}_{(-180^\circ,r>0dB)}$}
\psfrag{wc}{\scriptsize $\omega=1.16$}
\psfrag{-80}{\scriptsize $-\infty$}
\includegraphics[scale=.5]{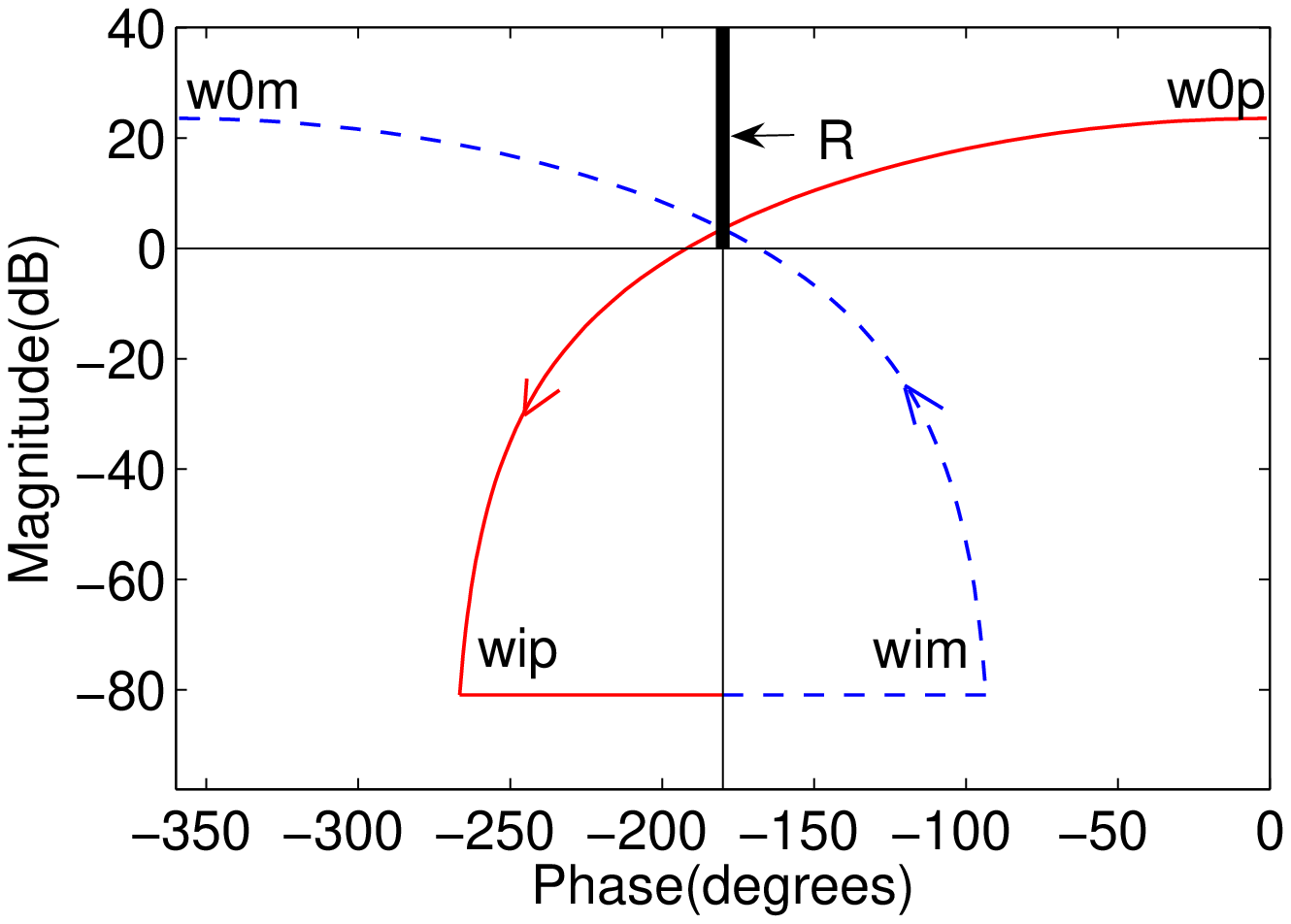}}
\caption{Nyquist and Nichols plots of Example \ref{MP_Stable_Ex}.}
\label{Fig:Stable_MP_NC_Ex}
\end{figure}

\begin{example}
\label{MP_Stable_Ex} Consider the feedback system Fig. \ref{Fig:1DOF_standard_Feedback_structure} with a stable, minimum
phase loop-function as given by:
\begin{equation} \label{Eq:Stable_MP_NC_Ex}
L(s)=\frac{K}{(\frac{s}{1}+1)(\frac{s}{2}+1)(\frac{s}{3}+1)}.\end{equation}
Determine the system stability for $K=5$ and $K=15$ by using the crossing concept, Nyquist and single-sheeted Nichols plots.
\end{example}

\begin{solution}
The given loop-function does not have unstable mode, then $N_p=0$. Fig. \ref{SubFig:Stable_MP_Ny_Ex1} shows the Nyquist diagram of $L(s)$ for $K=5$. The number of crossings on the ray $R_{(-\infty,-1)}$ is $N=0$, meaning that the Nyquist diagram does not encircle the critical point $(-1,0)$. Using (\ref{Nyquist_Criterion_Eq}), $N_z=N+N_p=0+0=0$ and therefore, the characteristic function $Q(s)$ does not have any zero in the RHP. This implies that the feedback system is stable for $K=5$. Fig. \ref{SubFig:Stable_MP_NC_Ex1} shows the single-sheeted Nichols plot of $L(s)$ for $K=5$. The number of crossings on the ray $R_{(-180^\circ,r>0dB)}$ is $N=0$, meaning that the equivalent Nyquist diagram in the complex plane does not encircle the critical point $(-1,0)$ and the feedback system is stable for $K=5$.

Similar method is employed for when $K=15$. Fig. \ref{SubFig:Stable_MP_Ny_Ex2} shows the Nyquist diagram of $L(s)$ for $K=15$. The number of crossings of $R_{(-\infty,-1)}$ is $N=2$, meaning that the Nyquist diagram encircles the critical point $(-1,0)$ two times in clockwise direction. By using (\ref{Nyquist_Criterion_Eq}), $N_z=N+N_p=2+0=2$ implying that the characteristics function $Q(s)$ has two zeros in the RHP. Then, the feedback system has two unstable poles with $K=15$. Fig. \ref{SubFig:Stable_MP_NC_Ex2} shows single-sheeted Nichols plot of $L(s)$ for $K=15$. The number of crossings of $R_{(-180^\circ,r>0dB)}$ is $N=2$, meaning that the corresponding Nyquist diagram of $L(s)$ encircles the critical point $(-1,0)$ two times in clockwise direction in the complex plane. $N \neq N_p$ thus, the feedback system is unstable for $K=15$.
\end{solution}

\begin{example}
\label{MP_Unstable_Ex} Consider the feedback system Fig. \ref{Fig:1DOF_standard_Feedback_structure} with an unstable, minimum
phase loop-function as given by:
\begin{equation} \label{Eq:Unstable_MP_NC_Ex}
L(s)=\frac{K(\frac{s}{3}+1)(\frac{s}{5}+1)}{(\frac{s}{2}-1)(\frac{s}{4}-1)}.
\end{equation}
Determine the system stability for $K=5$ and $K=1$ by using the crossing concept, Nyquist and single-sheeted Nichols plots.
\end{example}

\begin{solution}
At this example, the loop-function has two unstable poles, then $N_p=2$. Fig. \ref{SubFig:Unstable_MP_Ny_Ex1} shows the Nyquist
diagram of $L(s)$ for $K=5$. The number of crossings of $R_{(-\infty,-1)}$ is $N=-2$, meaning that the Nyquist diagram
encircles the critical point $(-1,0)$ two times in counterclockwise direction. Using (\ref{Nyquist_Criterion_Eq}), $N_z=N+N_p=-2+2=0$, implying that and the system characteristic function $Q(s)$ does not have any pole in the RHP. Then, the feedback system is stable for $K=5$. Fig. \ref{SubFig:Unstable_MP_NC_Ex1} shows the
single-sheeted Nichosl chart of $L(s)$ for $K=5$. The number of crossings of $R_{(-180^\circ,r>0dB)}$ is $N=-2$, meaning that the Nyquist diagram encircles the critical point $(-1,0)$ two times in counterclockwise direction. Since $N_z=-N_p$ then, the feedback system is stable for $K=5$.

Fig. \ref{SubFig:Unstable_MP_Ny_Ex2} shows the Nyquist diagram of $L(s)$ for $K=1$. The number of crossings of $R_{(-\infty,-1)}$ is $N=0$. Therefore, the Nyquist diagram does not encircle the critical point $(-1,0)$. By using (\ref{Nyquist_Criterion_Eq}), $N_z=N+N_p=0+2=2$, implying that the feedback system is unstable for $K=1$ with two RHP poles. Fig. \ref{SubFig:Unstable_MP_NC_Ex2} shows the Nyquist diagram of $L(s)$ for $K=1$. The number of crossings of $R_{(-\infty,-1)}$ is $N=0$, meaning that the Nyquist diagram does not encircle the critical point $(-1,0)$. Thus, $N\neq N_p$ and the feedback system is unstable for $K=1$.
\end{solution}

\begin{figure}
\centering
\psfrag{Im}[c][c]{\scriptsize  $Im\{L(s)\}$}
\psfrag{Re}{\scriptsize  $Re\{L(s)\}$}
\psfrag{w0}{\scriptsize$ \omega=0$}
\psfrag{win}{\scriptsize $\omega=\pm \infty$}
\psfrag{xin}{\scriptsize  $0$}
\psfrag{w0p}[r][r]{\scriptsize$\omega=0^{+}$}
\psfrag{w0m}{\scriptsize $\omega=0^{-}$}
\psfrag{wip}{\scriptsize $\omega=+\infty$}
\psfrag{wim}{\scriptsize $\omega=-\infty$}
\subfigure[K=5]{\label{SubFig:Unstable_MP_Ny_Ex1}
\psfrag{w0}{\scriptsize$ \omega=0$}
\psfrag{wc}[c][c]{\scriptsize $\omega=3.52$}
\psfrag{win}{\scriptsize $\omega=\infty$}
\psfrag{x0}{\scriptsize $5$}
\psfrag{xc}[c][c]{\scriptsize -3.52}
\psfrag{xin}{\scriptsize $2.67$}
\psfrag{R}{\scriptsize $\mathcal{R}_{(-\infty,-1)}$}
\includegraphics[scale=.5]{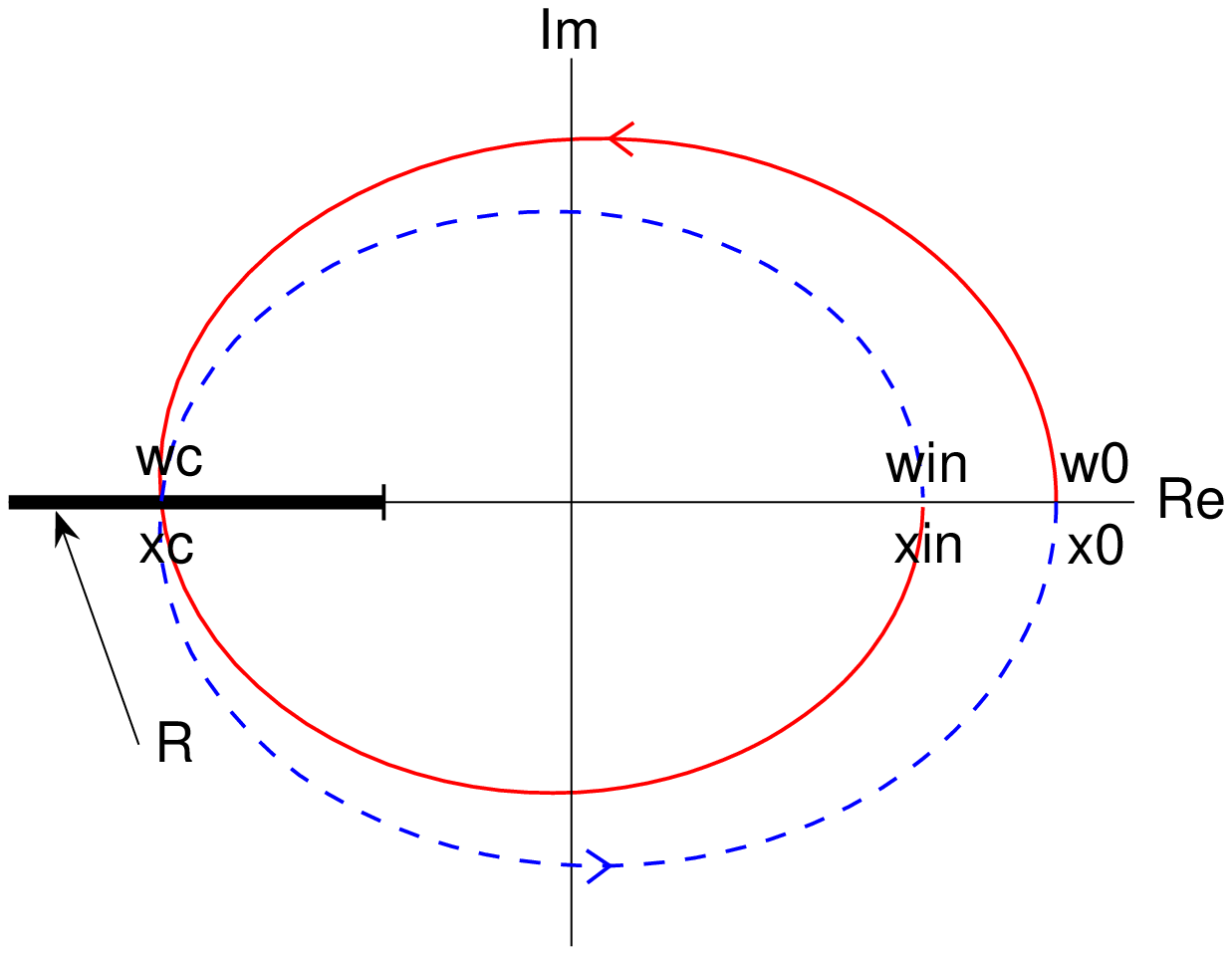}}
\hspace{1em}
\subfigure[K=5]{\label{SubFig:Unstable_MP_NC_Ex1}
\psfrag{R}{\scriptsize  $\mathcal{R}_{(-180^\circ,r>0dB)}$}
\psfrag{wc}{\scriptsize $\omega=1.16$}
\psfrag{-80}{\scriptsize $-\infty$}
\includegraphics[scale=.5]{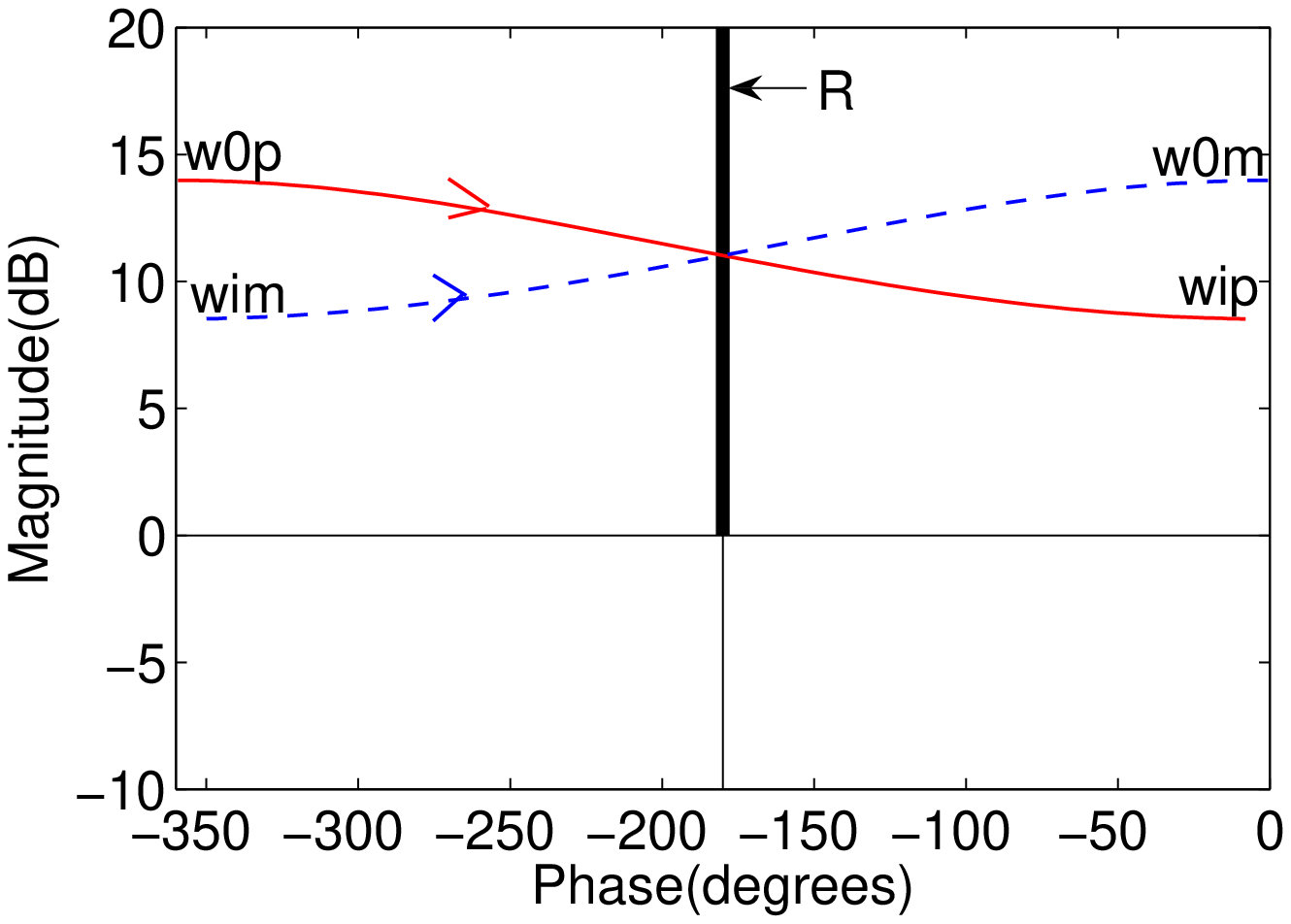}}
\\
\subfigure[K=1]{\label{SubFig:Unstable_MP_Ny_Ex2}
\psfrag{w0}{\scriptsize $\omega=0$}
\psfrag{wc}[c][c]{\scriptsize $\omega=3.5$}
\psfrag{win}{\scriptsize $\omega= \infty$}
\psfrag{x0}{\scriptsize $1$}
\psfrag{xc}[c][c]{\scriptsize -0.70}
\psfrag{xin}{\scriptsize $0.53$}
\psfrag{R}{\scriptsize $\mathcal{R}_{(-\infty,-1)}$}
\includegraphics[scale=.5]{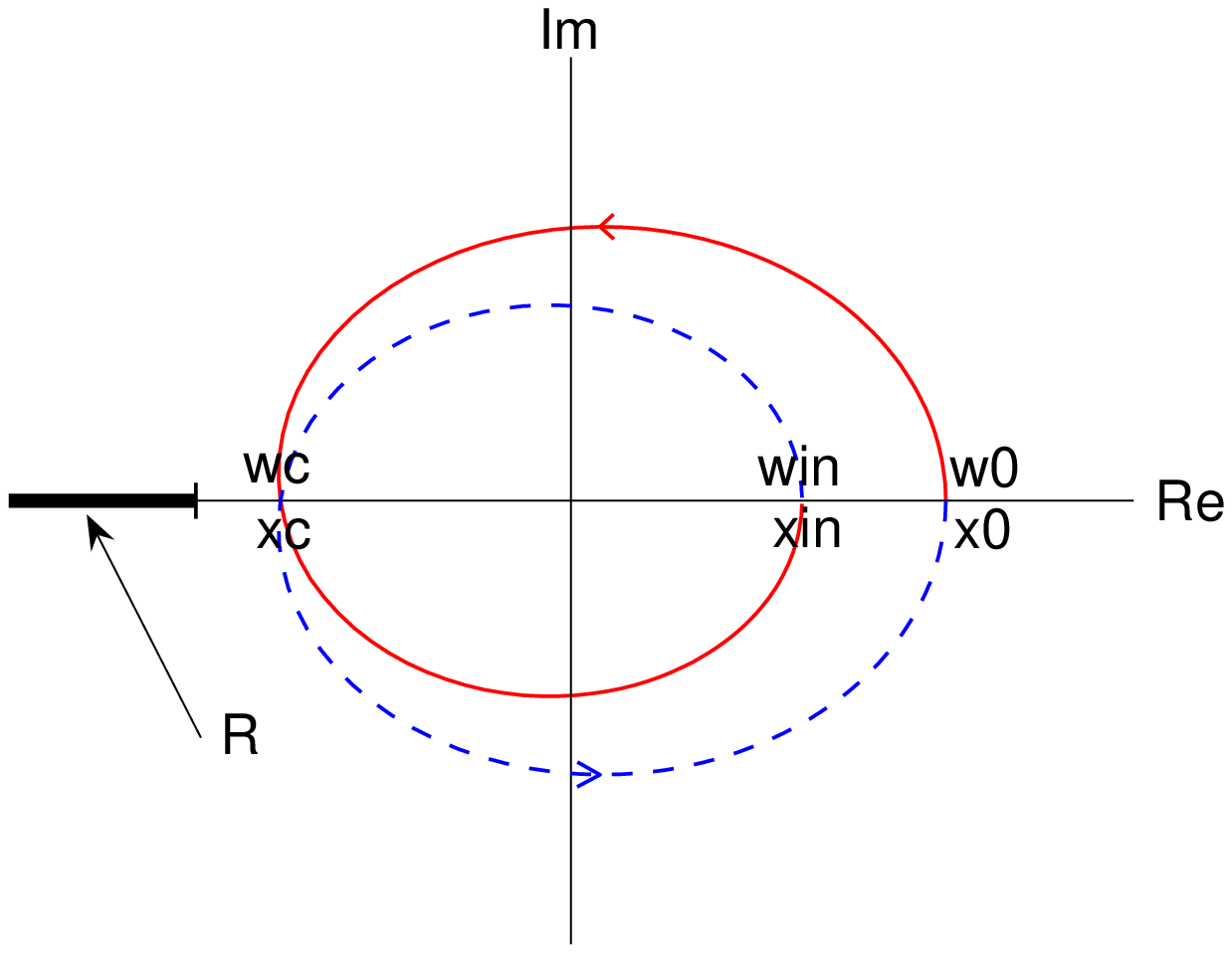}}
\hspace{1em}
\subfigure[K=1]{\label{SubFig:Unstable_MP_NC_Ex2}
\psfrag{R}{\scriptsize  $\mathcal{R}_{(-180^\circ,r>0dB)}$}
\psfrag{wc}{\scriptsize $\omega=1.16$}
\psfrag{-80}{\scriptsize $-\infty$}
\includegraphics[scale=.5]{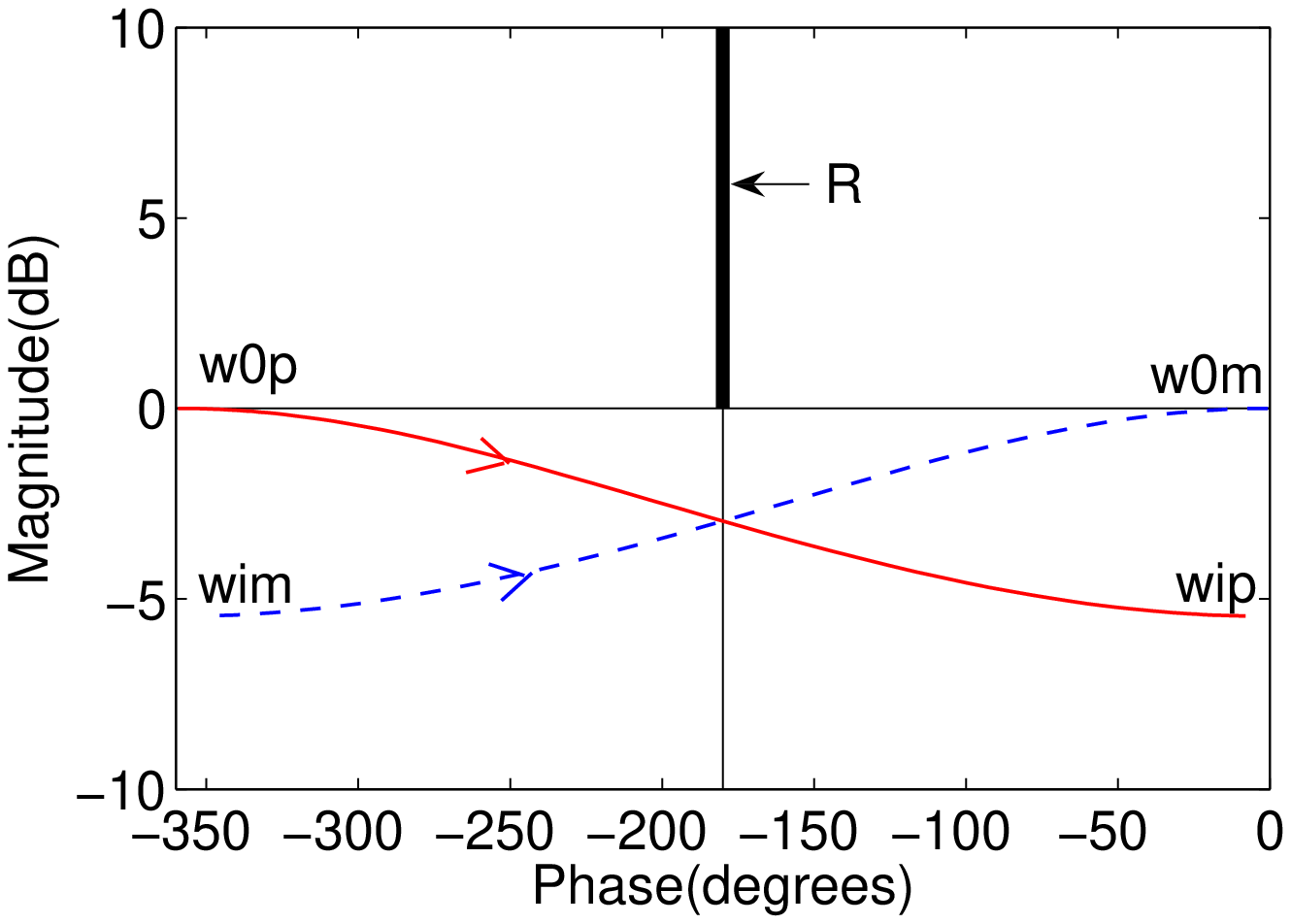}}
\caption{Nyquist plot of Example \ref{MP_Unstable_Ex}.}
     \label{Fig:Unstable_MP_Ny_Ex}
\end{figure}

\begin{example}
\label{NMP_Stable_Ny_Ex}Consider the feedback system Fig. \ref{Fig:1DOF_standard_Feedback_structure} with a stable but non-minimum phase loop-function as given by:
\begin{equation} \label{Eq:Stable_NMP_Ny_Ex}
L(s)=\frac{K(\frac{s}{0.5}-1)}{(\frac{s}{2}+1)(\frac{s}{3}+1)}\end{equation}
Determine the system stability for $K=0.5$ and $K=1.5$ by using the crossing concept, Nyquist and Nichols plots.
\end{example}
\begin{solution}
The given loop-function is stable, then $N_p=0$. According to the Nyquist diagram of $L(s)$ as shown Fig. \ref{SubFig:Stable_NMP_Ny_Ex1}, the number of crossings of $R_{(-\infty,-1)}$ is $0$ for $K=0.5$. It means that the Nyquist diagram does not encircle the critical point $(-1,0)$. By using (\ref{Nyquist_Criterion_Eq}), $N_z=N+N_p=0+0=0$, which implies the feedback system is stable for $K=0.5$. As Fig. \ref{SubFig:Stable_NMP_NC_Ex1} shows, the number of crossings of
$R_{(-180^\circ,r>0dB)}$ is $0$ for $K=0.5$ and then $N=0$. Since $N=-N_p=0$, the feedback system is stable for $K=0.5$.

Subsequently, Fig. \ref{SubFig:Stable_NMP_Ny_Ex2} shows the Nyquist diagram of $L(s)$ for $K=1.5$. The number of crossings of
$R_{(-\infty,-1)}$ is $N=+1$, meaning that the Nyquist diagram encircles the critical point $(-1,0)$ once in the clockwise
direction. By using (\ref{Nyquist_Criterion_Eq}), $N_z=N+N_p=1+0=1$, showing that the feedback system is unstable for $K=1.5$. As shown in Fig. \ref{SubFig:Stable_NMP_NC_Ex2}, the number of crossing of the ray $R_{(-180^\circ,r>0dB)}$ by the single-sheeted Nichols plot of $L(s)$ will be $N=1$. Thus, $N  \neq N_p$ and the feedback system is unstable for $K=1.5$.
\end{solution}

\begin{figure}
\centering
\psfrag{Im}[c][c]{\scriptsize  $Im\{L(s)\}$}
\psfrag{Re}{\scriptsize  $Re\{L(s)\}$}
\psfrag{w0}{\scriptsize$ \omega=0$}
\psfrag{win}{\scriptsize $\omega=\pm \infty$}
\psfrag{xin}{\scriptsize  $0$}
\psfrag{w0p}[r][r]{\scriptsize$\omega=0^{+}$}
\psfrag{w0m}{\scriptsize $\omega=0^{-}$}
\psfrag{wip}{\scriptsize $\omega=+\infty$}
\psfrag{wim}{\scriptsize $\omega=-\infty$}
\subfigure[K=0.5]{\label{SubFig:Stable_NMP_Ny_Ex1}
\psfrag{w0}{\scriptsize$ \omega=0$}
\psfrag{wc}[c][c]{\scriptsize $\omega=3.33$}
\psfrag{win}{\scriptsize $\omega=\infty$}
\psfrag{x0}{\scriptsize $-0.5$}
\psfrag{xc}[c][c]{\scriptsize 1.19}
\psfrag{xin}{\scriptsize $0$}
\psfrag{R}{\scriptsize $\mathcal{R}_{(-\infty,-1)}$}
\includegraphics[scale=.5]{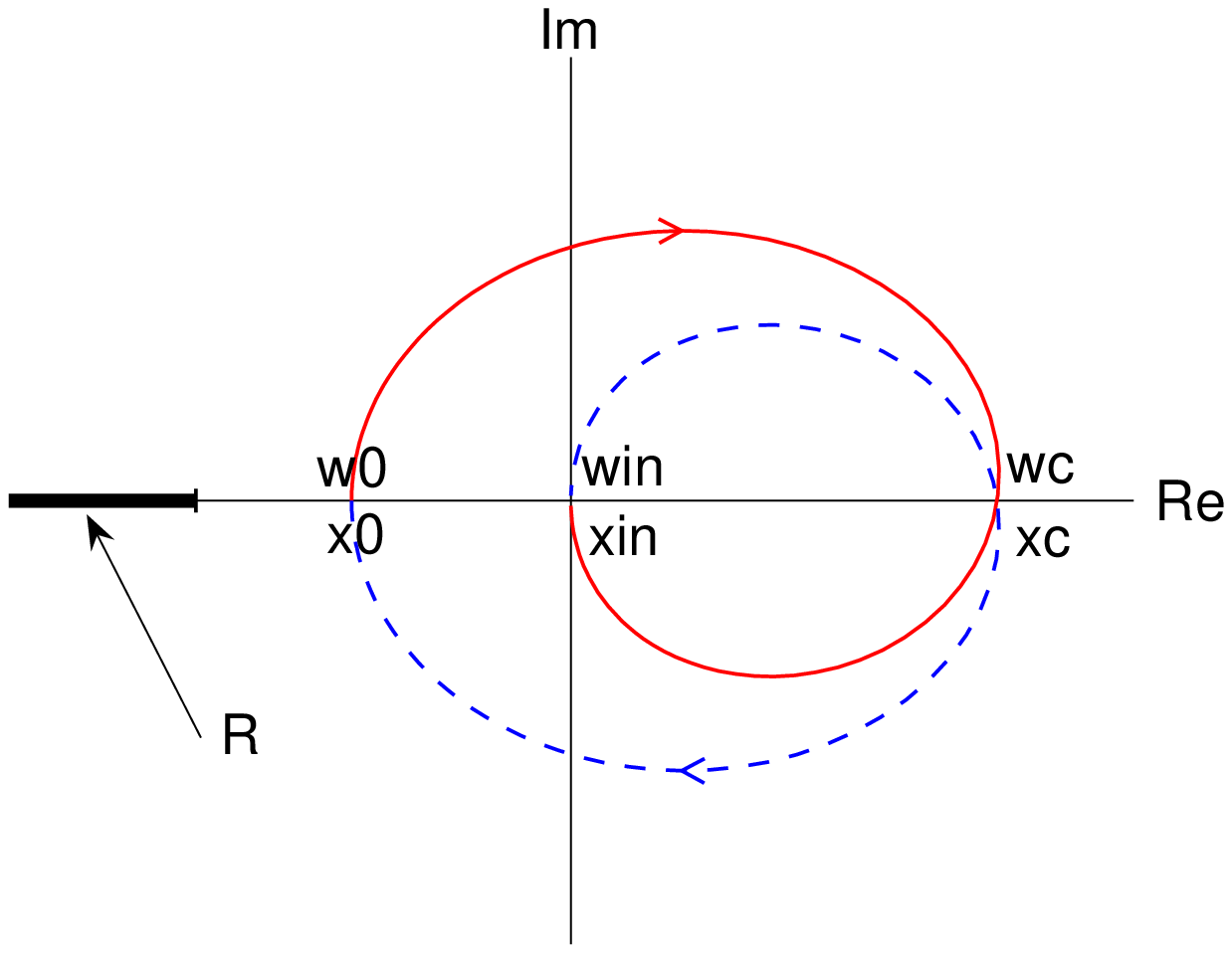}}
\hspace{1em}
\subfigure[K=0.5]{\label{SubFig:Stable_NMP_NC_Ex1}
\psfrag{R}{\scriptsize  $\mathcal{R}_{(-180^\circ,r>0dB)}$}
\psfrag{wc}{\scriptsize $\omega=1.16$}
\psfrag{-25}{\scriptsize $-\infty$}
\includegraphics[scale=.5]{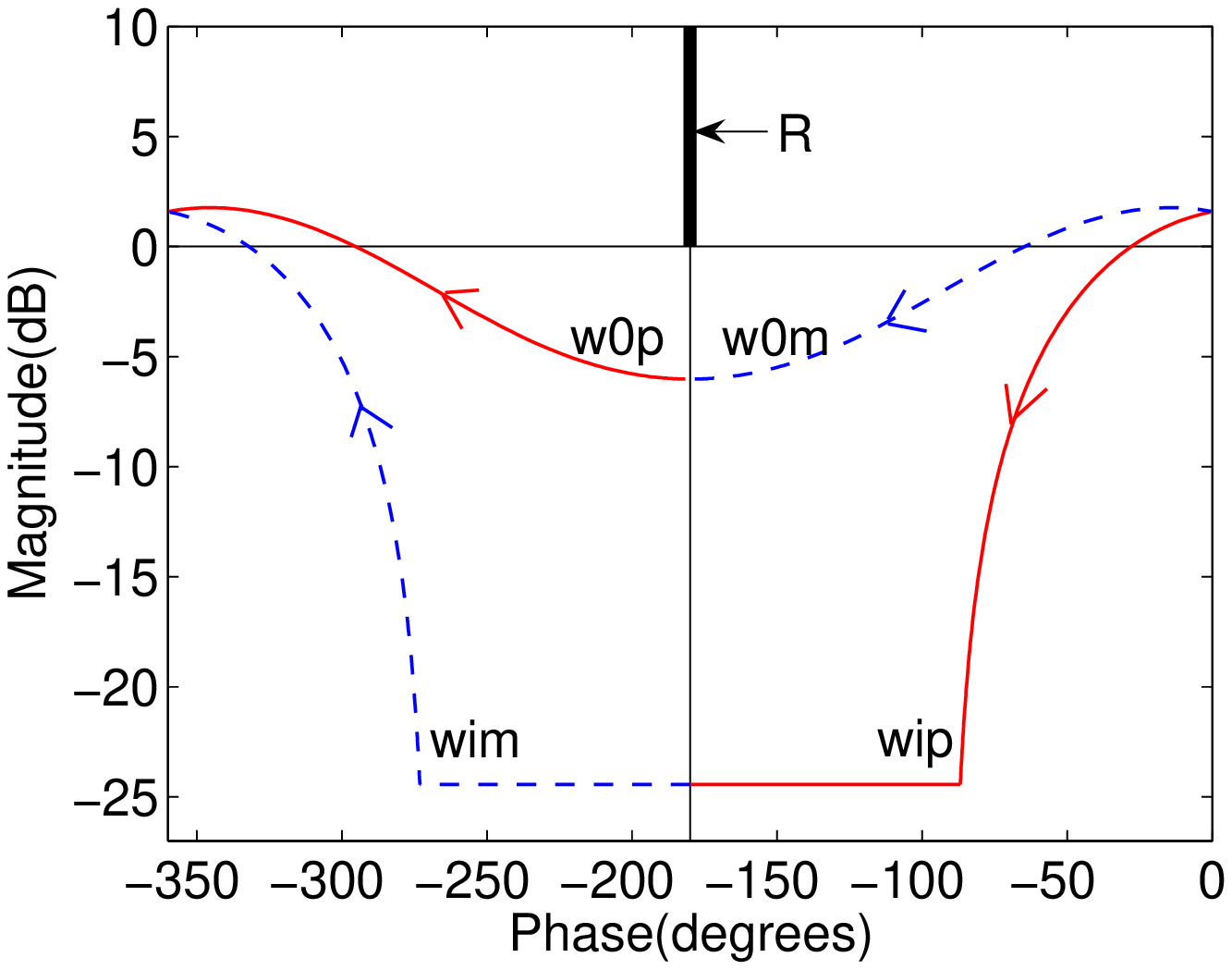}}\\
\subfigure[K=1.5]{\label{SubFig:Stable_NMP_Ny_Ex2}
\psfrag{w0}{\scriptsize $\omega=0$}
\psfrag{wc}[c][c]{\scriptsize $\omega=3.32$}
\psfrag{win}{\scriptsize $\omega= \infty$}
\psfrag{x0}[c][c]{\scriptsize $-1.5$}
\psfrag{xc}[c][c]{\scriptsize 3.58}
\psfrag{xin}{\scriptsize $0$}
\psfrag{R}{\scriptsize $\mathcal{R}_{(-\infty,-1)}$}
\includegraphics[scale=.5]{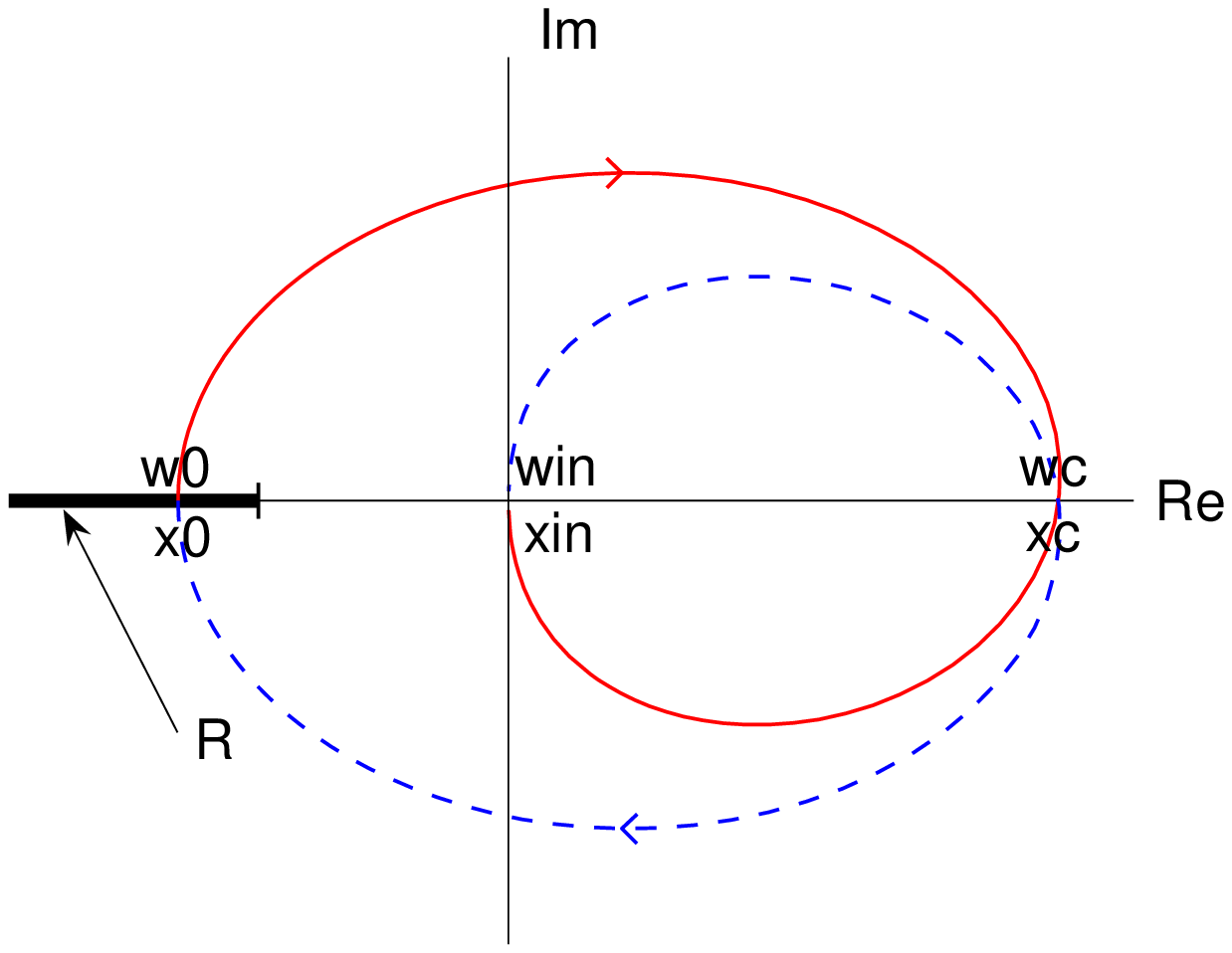}}
\hspace{1em}
\subfigure[K=1.5]{\label{SubFig:Stable_NMP_NC_Ex2}
\psfrag{wc}{\scriptsize $\omega=1.16$}
\psfrag{-20}{\scriptsize $-\infty$}
\psfrag{R}{\scriptsize  $\mathcal{R}_{(-180^\circ,r>0dB)}$}
\includegraphics[scale=.5]{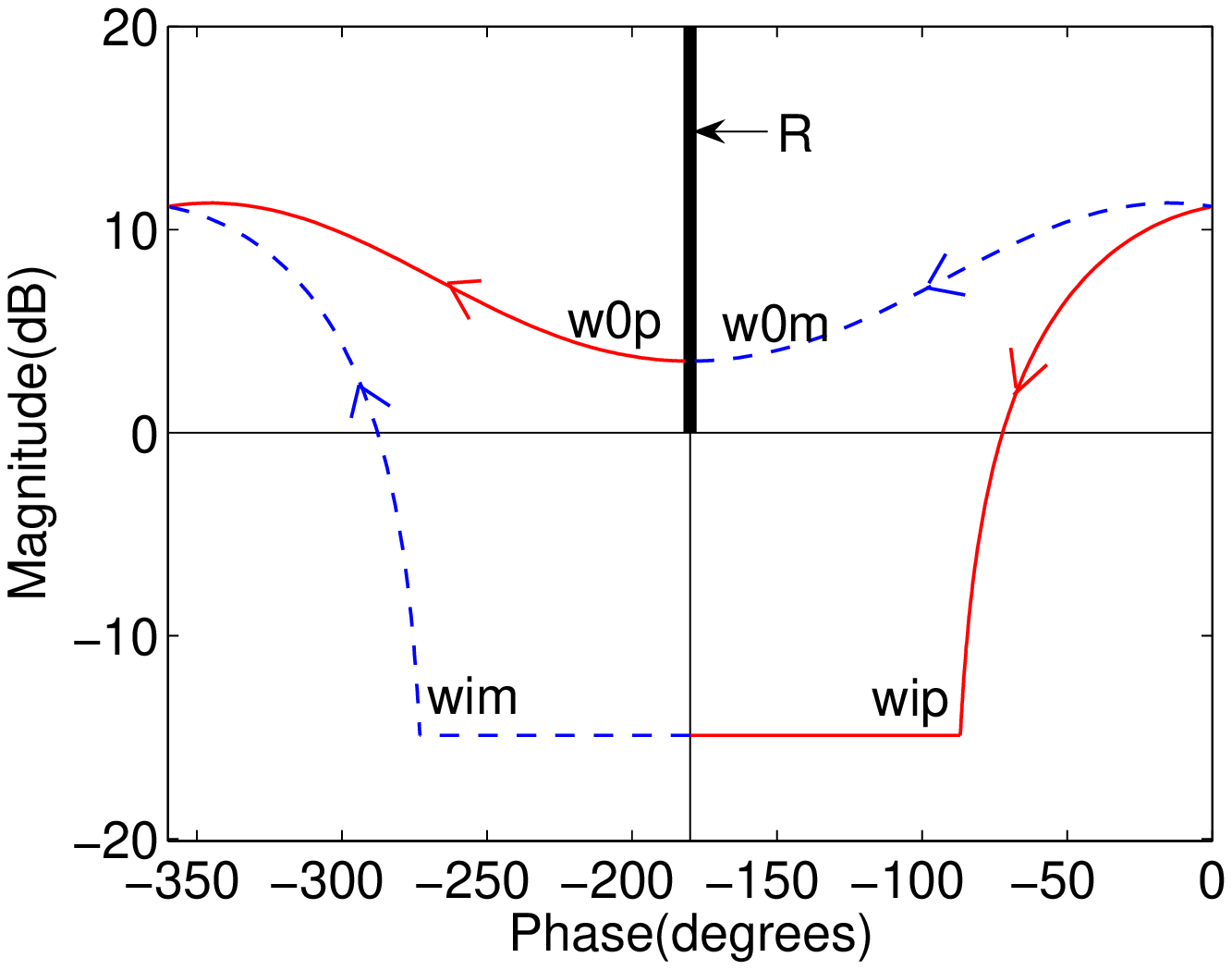}}
\caption{Nyquist plot of Example \ref{NMP_Stable_Ny_Ex}.}
\label{Fig:Stable_NMP_Ny_Ex}
\end{figure}


\begin{figure}
\centering
\psfrag{Im}[c][c]{\scriptsize  $Im\{L(s)\}$}
\psfrag{Re}{\scriptsize  $Re\{L(s)\}$}
\psfrag{w0}{\scriptsize$ \omega=0$}
\psfrag{win}{\scriptsize $\omega=\pm \infty$}
\psfrag{xin}{\scriptsize  $0$}
\psfrag{w0p}[r][r]{\scriptsize$\omega=0^{+}$}
\psfrag{w0m}{\scriptsize $\omega=0^{-}$}
\psfrag{wip}{\scriptsize $\omega=+\infty$}
\psfrag{wim}{\scriptsize $\omega=-\infty$}
\subfigure[K=0.5]{\label{SubFig:Unstable_NMP_Ny_Ex1}
\psfrag{w0}{\scriptsize$ \omega=0$}
\psfrag{wc}[c][c]{\scriptsize $\omega=2.05$}
\psfrag{win}{\scriptsize $\omega=\infty$}
\psfrag{x0}[c][c]{\scriptsize $-0.5$}
\psfrag{xc}[c][c]{\scriptsize -1.2}
\psfrag{xin}{\scriptsize $0$}
\psfrag{R}{\scriptsize $\mathcal{R}_{(-\infty,-1)}$}
\includegraphics[scale=.5]{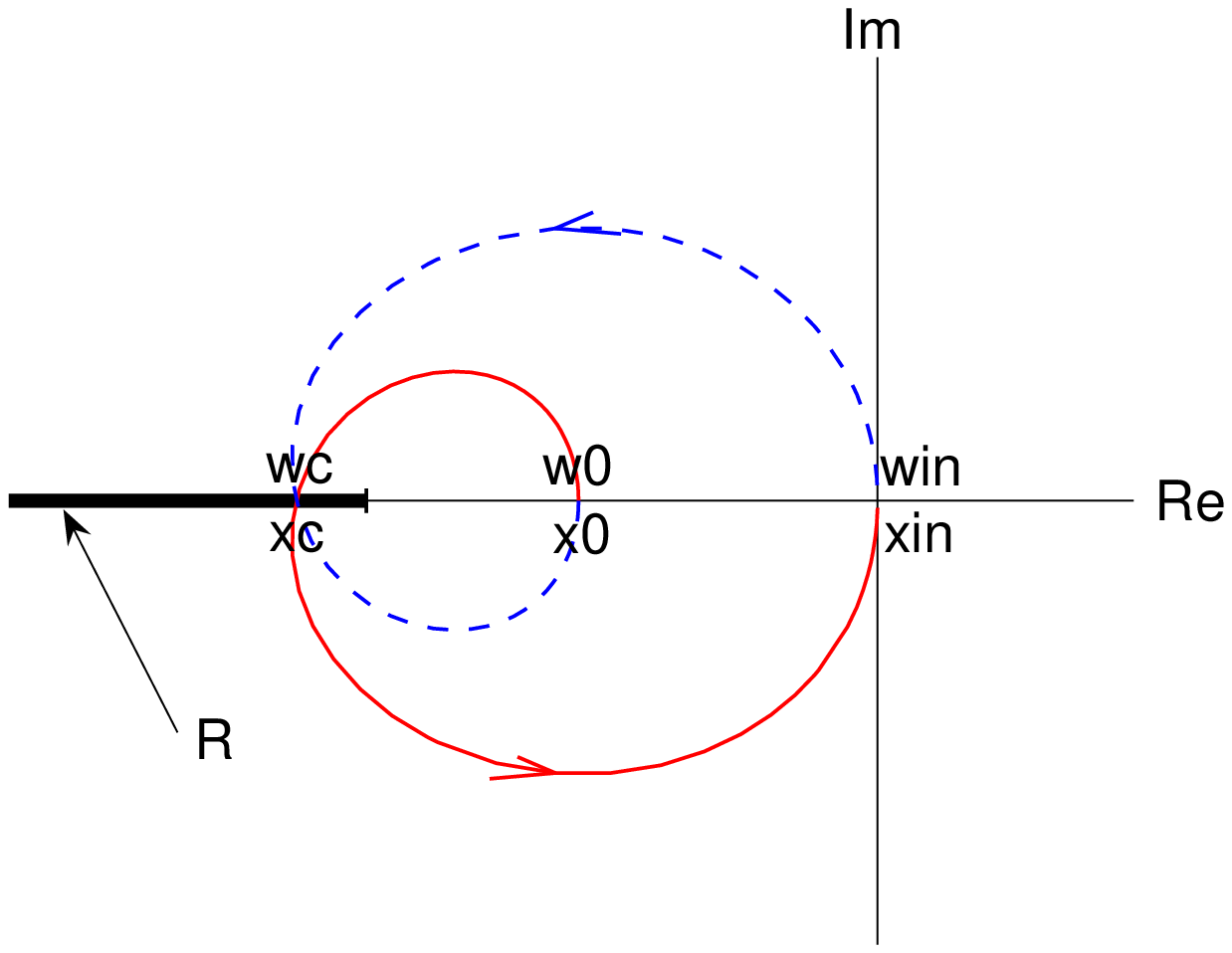}}
\hspace{1em}
\subfigure[K=0.5]{\label{SubFig:Unstable_NMP_NC_Ex1}
\psfrag{wc}{\scriptsize $\omega=1.16$}
\psfrag{-25}{\scriptsize $-\infty$}
\psfrag{R}{\scriptsize  $\mathcal{R}_{(-180^\circ,r>0dB)}$}
\includegraphics[scale=.5]{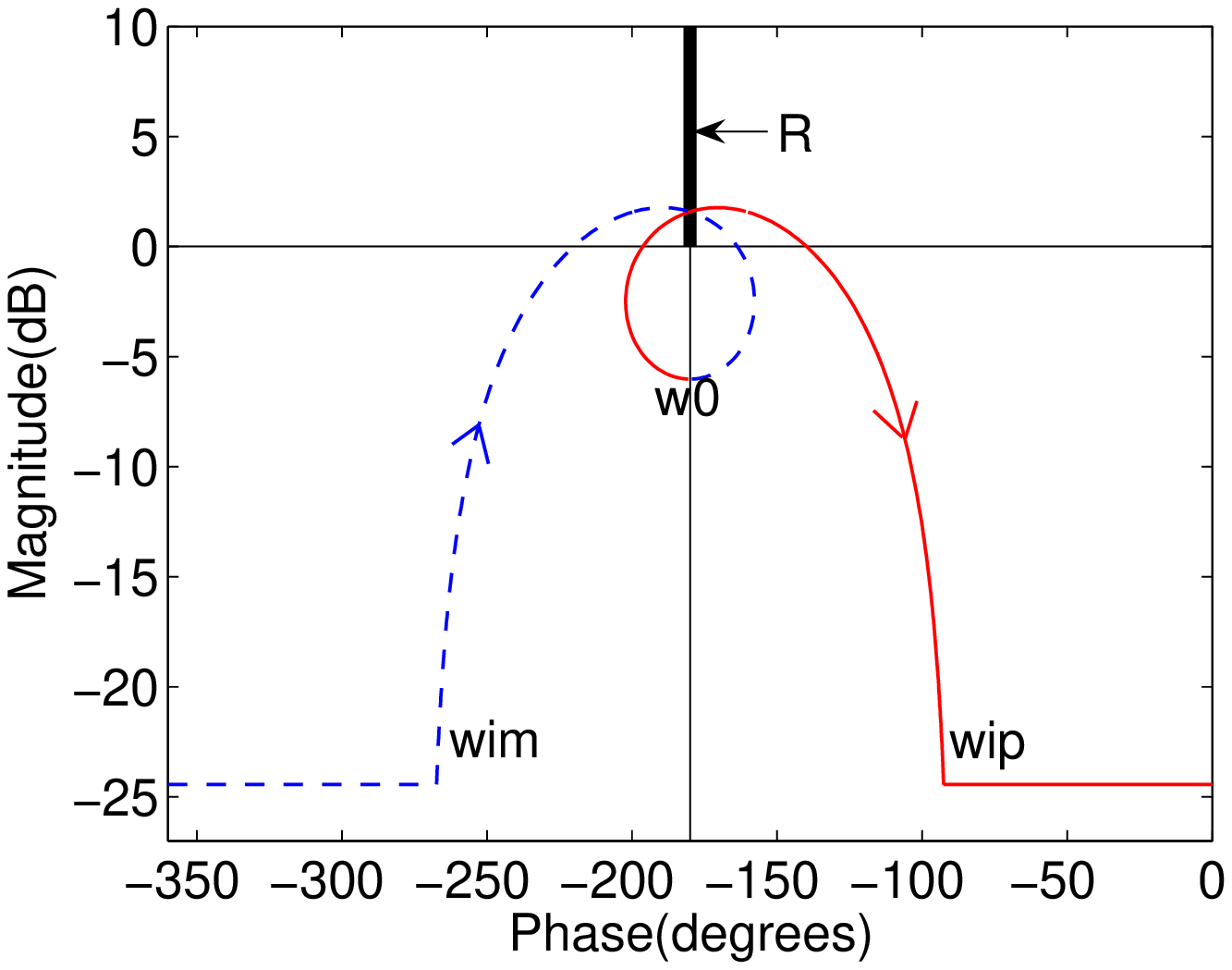}}
\\
\subfigure[K=1.5]{\label{SubFig:Unstable_NMP_Ny_Ex2}
\psfrag{w0}{\scriptsize $\omega=0$}
\psfrag{wc}[c][c]{\scriptsize $\omega=2.06$}
\psfrag{win}{\scriptsize $\omega= \infty$}
\psfrag{x0}[c][c]{\scriptsize $-1.5$}
\psfrag{xc}[c][c]{\scriptsize-3.58}
\psfrag{xin}{\scriptsize $0$}
\psfrag{R}{\scriptsize$\mathcal{R}_{(-\infty,-1)}$}
\includegraphics[scale=.5]{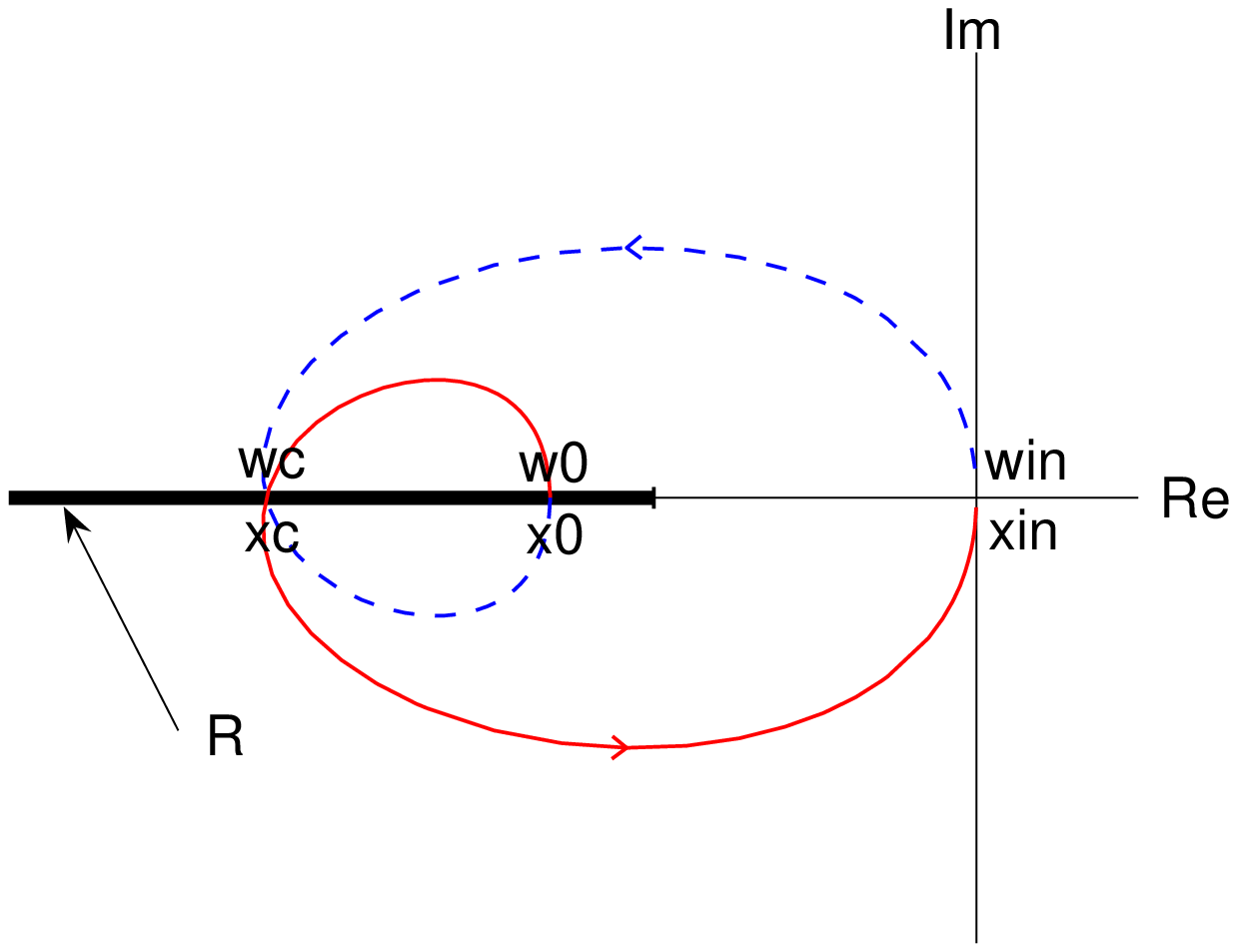}}
\hspace{1em}
\subfigure[K=1.5]{\label{SubFig:Unstable_NMP_NC_Ex2}
\psfrag{wc}{\scriptsize $\omega=1.16$}
\psfrag{-15}{\scriptsize$-\infty$}
\psfrag{R}{\scriptsize  $\mathcal{R}_{(-180^\circ,r>0dB)}$}
\includegraphics[scale=.5]{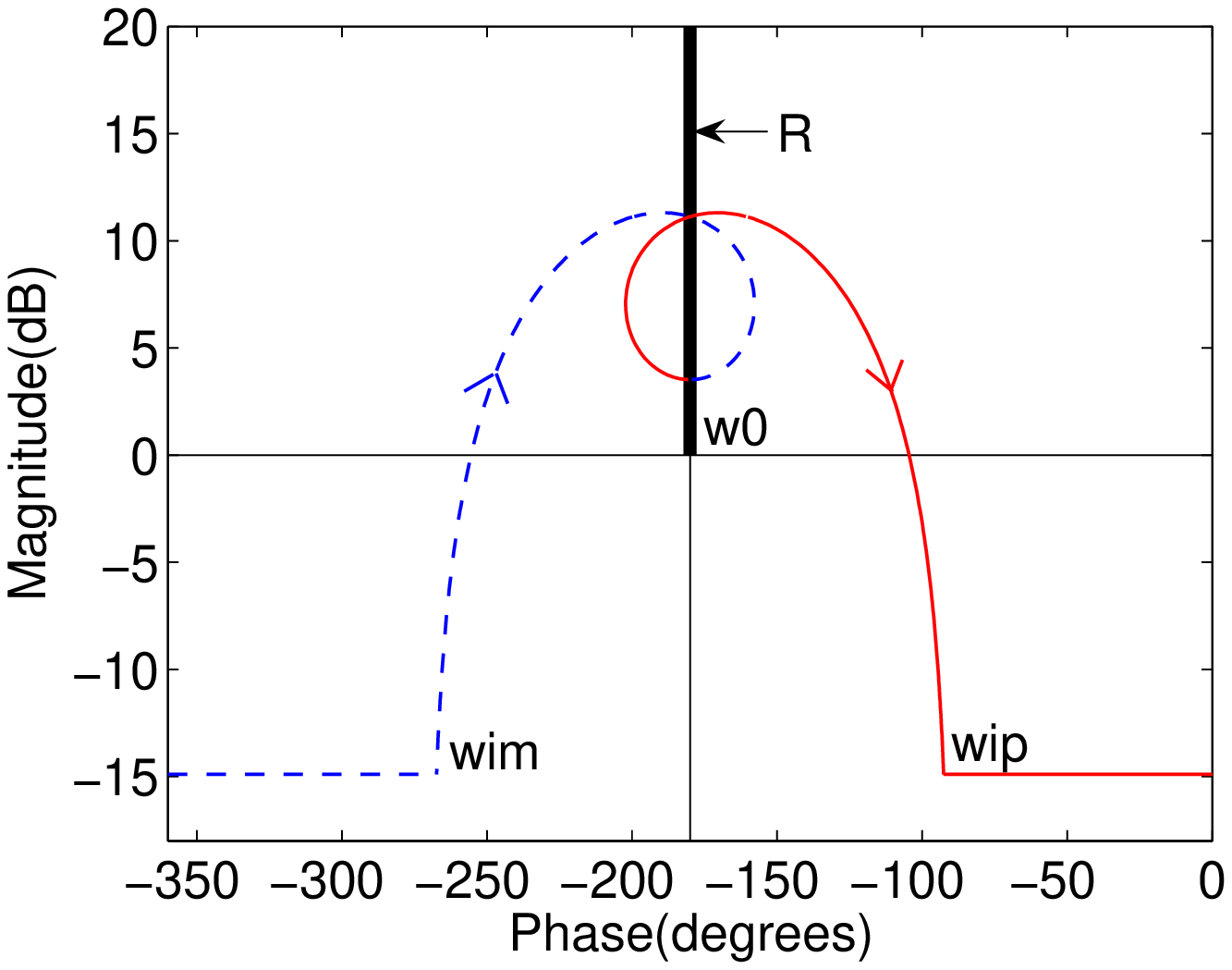}}
\caption{Nyquist plot of Example \ref{NMP_UnStable_Ny_Ex}.}
     \label{Fig:Unstable_NMP_Ny_Ex}
\end{figure}

\begin{example}
\label{NMP_UnStable_Ny_Ex}Consider the feedback system Fig. \ref{Fig:1DOF_standard_Feedback_structure} with an unstable and
nonminimum phase loop-function as given by:
\begin{equation} \label{Eq:Unstable_NMP_Ny_Ex}
L(s)=\frac{K(\frac{s}{0.5}-1)}{(\frac{s}{2}-1)(\frac{s}{3}-1)}\end{equation}
Determine the system stability for $K=0.5$ and $K=1.5$ by using the crossing concept, Nyquist and Nichols plots.
\end{example}
\begin{solution}
The given loop-function has two unstable poles, then $N_p=2$. The Nyquist diagram of $L(s)$ for $K=0.5$ is shown in Fig. \ref{SubFig:Unstable_NMP_Ny_Ex1}. The number of crossings of $R_{(-\infty,-1)}$ is $-2$. It means that the Nyquist diagram encircles the critical point $(-1,0)$ two times in counterclockwise direction. Using (\ref{Nyquist_Criterion_Eq}), $N_z=N+N_p=-2+2=0$ which implies the feedback system does not have any pole at the RHP for $K=0.5$ and is therefore stable. The single-sheeted Nichols plot of $L(s)$ shows two positive crossings that occur on $R_{(-180^\circ,r>0dB)}$, for $K=0.5$, that is $N=-2$. Thus, $N=-N_p$ and the feedback system is stable for $K=0.5$.

Fig. \ref{SubFig:Unstable_NMP_Ny_Ex2} shows the Nyquist diagram of $L(s)$ for $K=1.5$. At this example, there are three crossings of $R_{(-\infty,-1)}$, one positive and two negative, so $N=-1$. It means that the Nyquist diagram encircles the critical point $(-1,0)$ once in the counterclockwise direction. Again, using (\ref{Nyquist_Criterion_Eq}), $N_z=N+N_p=-1+2=1$, saying that the feedback system has one RHP pole and them is unstable for $K=1.5$. Fig. \ref{SubFig:Unstable_NMP_NC_Ex2} shows the single-sheeted Nichols plot of $L(s)$ for $K=1.5$. At this example, there are
three crossings of $R_{(-\infty,-1)}$, two negative and one positive which lead to $N=+1$. Since $N \neq N_p$, then the feedback system is unstable for $K=1.5$.
\end{solution}

\begin{figure}
\centering
\psfrag{Im}[c][c]{\scriptsize  $Im\{L(s)\}$}
\psfrag{Re}{\scriptsize  $Re\{L(s)\}$}
\psfrag{w0}{\scriptsize$ \omega=0$}
\psfrag{win}{\scriptsize $\omega=\pm \infty$}
\psfrag{xin}{\scriptsize  $0$}
\psfrag{w0p}[r][r]{\scriptsize$\omega=0^{+}$}
\psfrag{w0m}{\scriptsize $\omega=0^{-}$}
\psfrag{wip}{\scriptsize $\omega=+\infty$}
\psfrag{wim}{\scriptsize $\omega=-\infty$}
\subfigure[K=1]{\label{SubFig:Integral_Stable_MP_Ny_Ex1}
\psfrag{w0}{\scriptsize$ \omega=0$}
\psfrag{wc}[c][c]{\scriptsize $\omega=1.16$}
\psfrag{win}{\scriptsize $\omega=\infty$}
\psfrag{x0}{\scriptsize $Re\{L(s)\}=-2.5$}
\psfrag{xc}[c][c]{\scriptsize -0.39}
\psfrag{xin}{\scriptsize$0$}
\psfrag{inf}{\scriptsize $\infty$}
\psfrag{R}{\scriptsize $\mathcal{R}_{(-\infty,-1)}$}
\includegraphics[scale=.5]{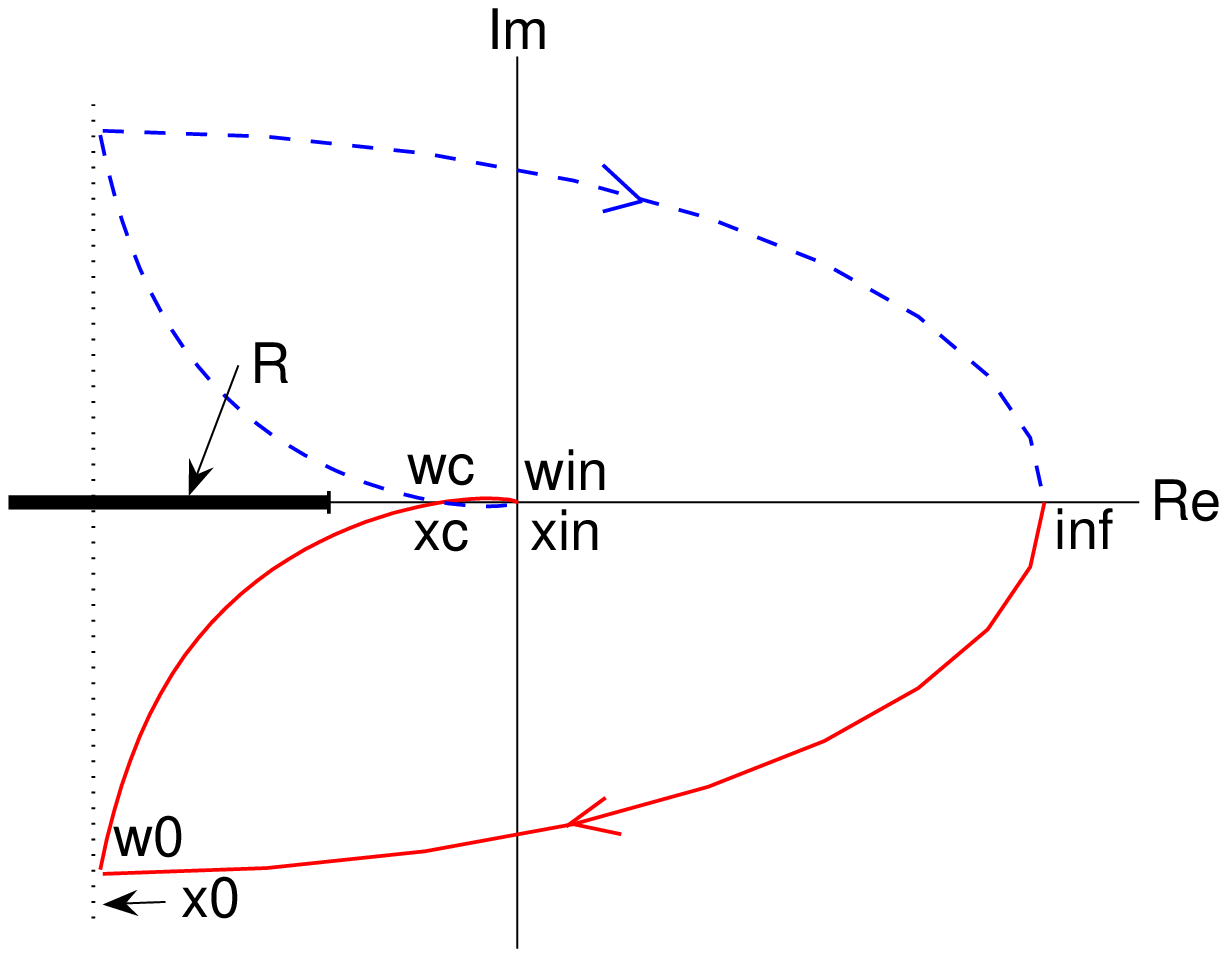}}
\hspace{1em}
\subfigure[K=1]{\label{SubFig:Integral_Stable_MP_NC_Ex1}
\psfrag{wc}{\scriptsize $\omega=1.16$}
\psfrag{-80}{\scriptsize $-\infty$}
\psfrag{R}{\scriptsize  $\mathcal{R}_{(-180^\circ,r>0dB)}$}
\includegraphics[scale=.5]{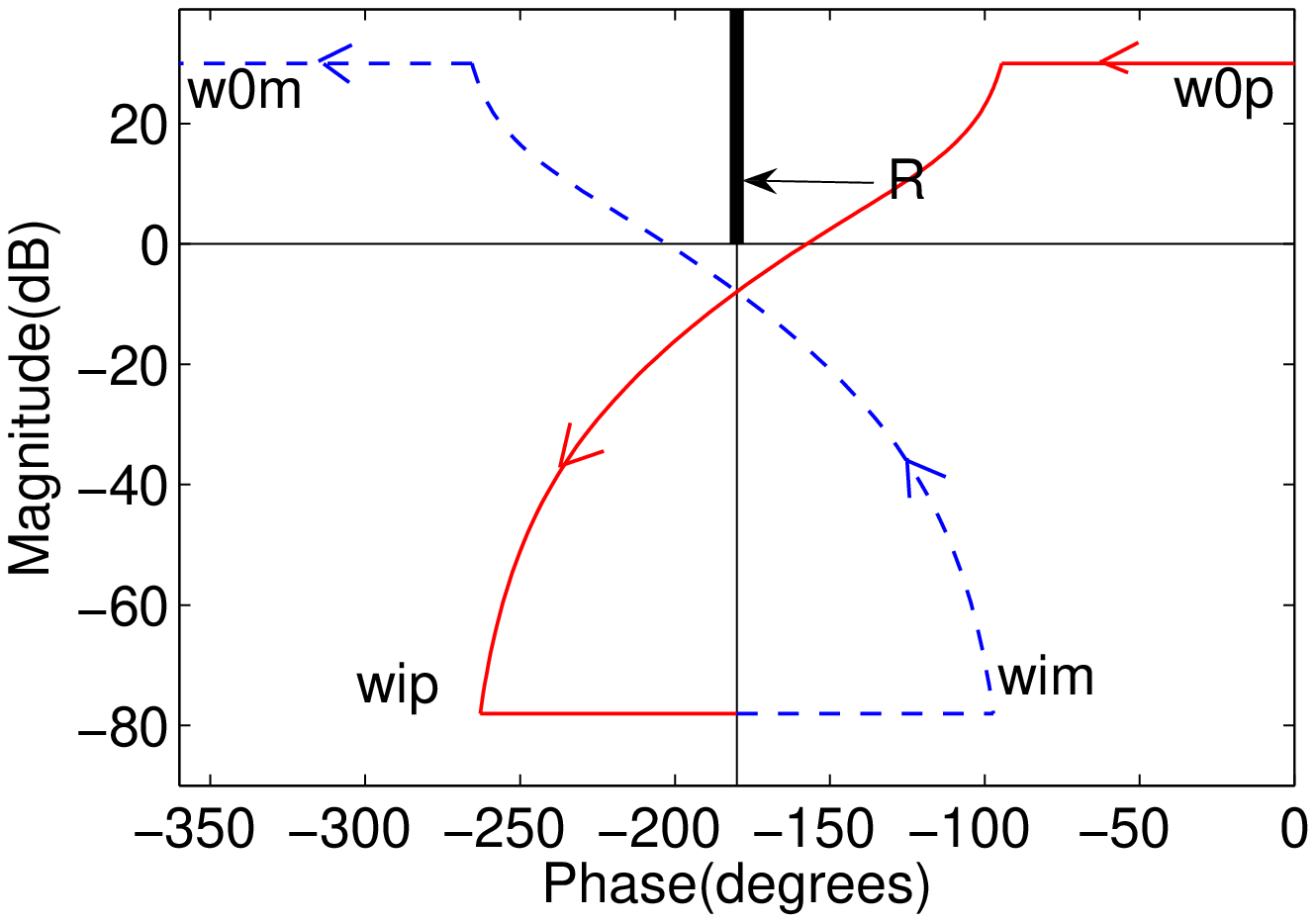}}
\\
\subfigure[K=5]{\label{SubFig:Integral_Stable_MP_Ny_Ex2}
\psfrag{w0}{\scriptsize$ \omega=0$}
\psfrag{wc}[c][c]{\scriptsize $\omega=1.15$}
\psfrag{win}{\scriptsize $\omega=\infty$}
\psfrag{x0}{\scriptsize $Re\{L(s)\}=-12.5$}
\psfrag{xc}[c][c]{\scriptsize -1.95}
\psfrag{xin}{\scriptsize $0$}
\psfrag{inf}{\scriptsize$\infty$}
\psfrag{R}{\scriptsize $\mathcal{R}_{(-\infty,-1)}$}
\includegraphics[scale=.5]{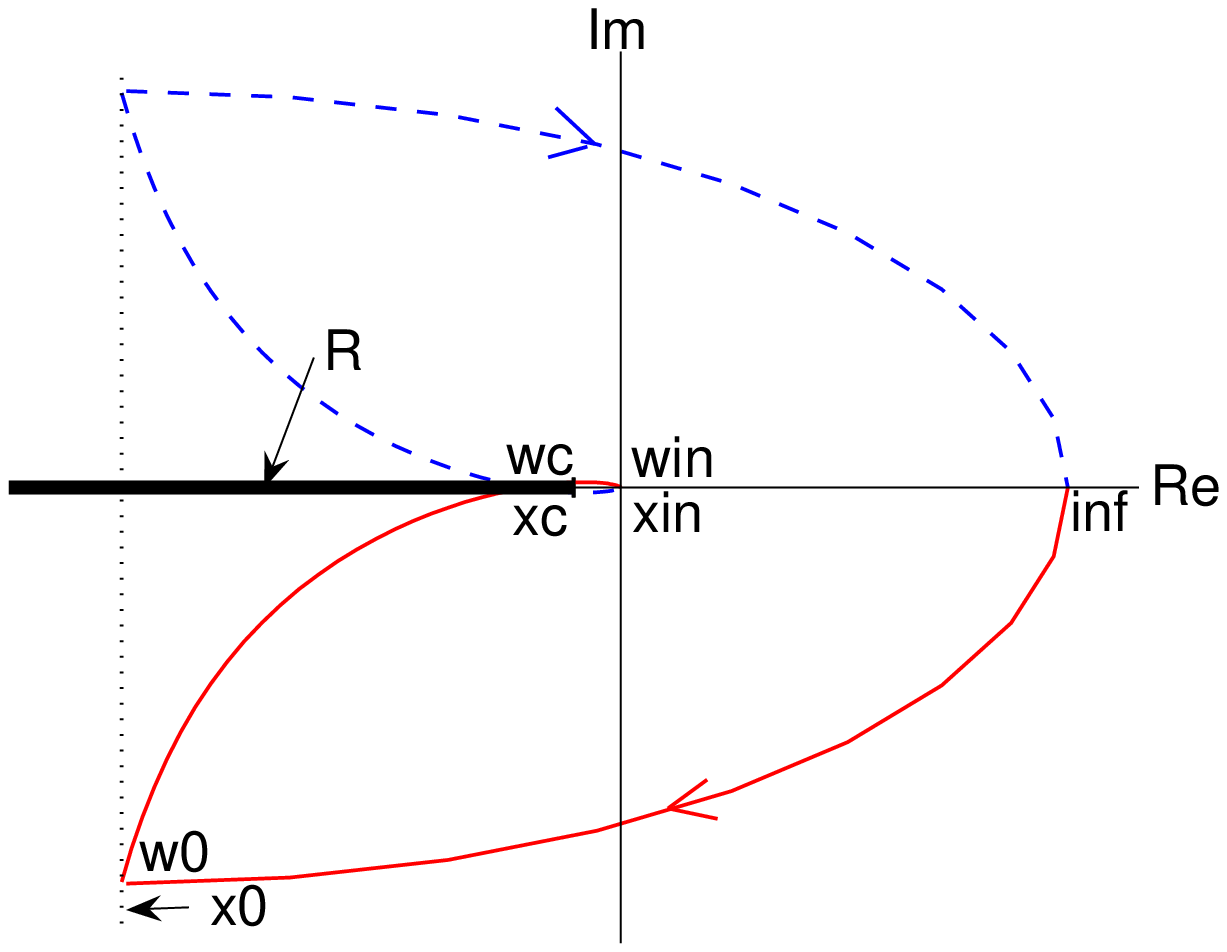}}
\hspace{1em}
\subfigure[K=5]{\label{SubFig:Intehral_Stable_MP_NC_Ex2}
\psfrag{R}{\scriptsize  $\mathcal{R}_{(-180^\circ,r>0dB)}$}
\psfrag{wc}{\scriptsize $\omega=1.16$}
\psfrag{-60}{\scriptsize $-\infty$}
\psfrag{40}[r][r]{\scriptsize $+\infty$}
\includegraphics[scale=.5]{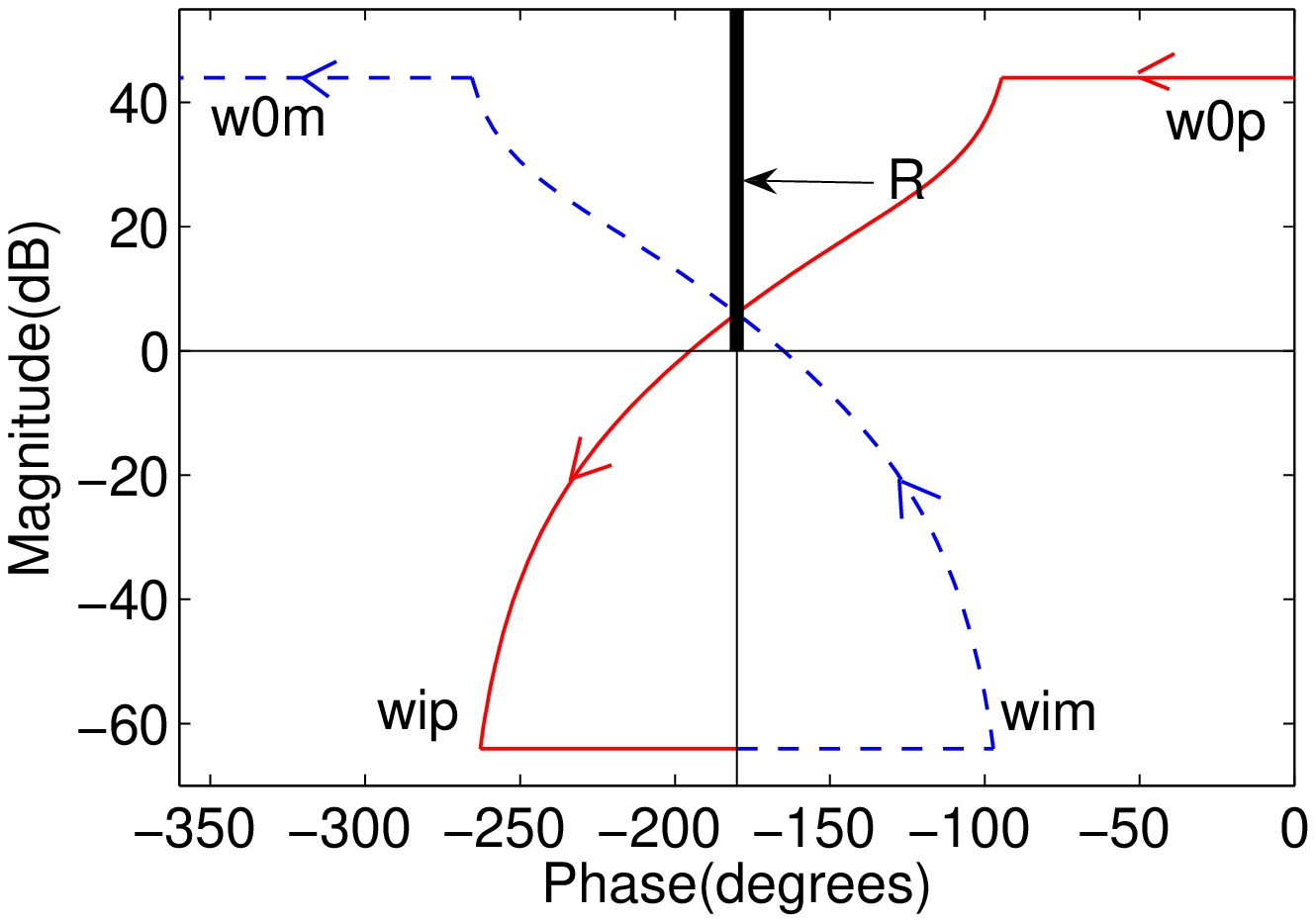}}
\caption{Nyquist plot of Example \ref{Integral_MP_Stable_Ny_Ex}.}
     \label{Fig:Integral_Stable_MP_Ny_Ex}
\end{figure}

\begin{example}
\label{Integral_MP_Stable_Ny_Ex} The use of integrator in the control structure is necessary if tracking is one of the desired objectives. It makes the tracking error to be zero. In this example, a feedback system with integrator is considered. Let $L(s)$ in feedback system Fig. \ref{Fig:1DOF_standard_Feedback_structure} be as follows:
\begin{equation} \label{Eq:Integral_Stable_MP_Ny_Ex}
L(s)=\frac{K}{s(\frac{s}{0.5}+1)(\frac{s}{2}+1)}\end{equation}
Using the crossing concept, determine the system stability for $K=1$ and $K=5$.
\end{example}
\begin{solution}
The loop-function has a pole at the origin. As mentioned earlier, to plot the Nyquist diagram which is the map of the standard Nyquist contour through $L(s)$ into the complex plan, $\Gamma$ must detour around the poles on the imaginary axis. Thus, $N_p=0$ as $L(s)$ does not have any RHP pole with the real part greater than zero. According to the Nyquist diagram of $L(s)$ as shown Fig. \ref{SubFig:Integral_Stable_MP_Ny_Ex1}, the number of crossings of $R_{(-\infty,-1)}$ is $0$ for $K=1$. It means that the Nyquist diagram does not encircle the critical point $(-1,0)$. Using (\ref{Nyquist_Criterion_Eq}), $N_z=N+N_p=0+0=0$ which implies the feedback system is stable for $K=1$. The single-sheeted Nichols plot of $L(s)$, as shown Fig.
\ref{SubFig:Integral_Stable_MP_NC_Ex1}, demonstrates that there is no crossing and $N=0$. Then, $N \neq -N_p=0$ and the feedback system is stable for $K=1$.

Subsequently, Fig. \ref{SubFig:Integral_Stable_MP_Ny_Ex2} shows the Nyquist diagram of $L(s)$ for $K=5$. The number of crossings of $R_{(-\infty,-1)}$ is $N=+2$, meaning that the Nyquist diagram encircles the critical point $(-1,0)$ once in the clockwise
direction. By using (\ref{Nyquist_Criterion_Eq}), $N_z=N+N_p=2+0=2$, saying that the feedback system has two unstable poles and is
unstable for $K=5$. As it is shown by Fig. \ref{SubFig:Intehral_Stable_MP_NC_Ex2}, the sum of crossing is $N=2$ for $K=5$. Since $N \neq -N_p$, then the feedback system is unstable for $K=5$.
\end{solution}

\begin{figure}
\centering
\psfrag{Im}[c][c]{\scriptsize  $Im\{L(s)\}$}
\psfrag{Re}{\scriptsize  $Re\{L(s)\}$}
\psfrag{w0}{\scriptsize$ \omega=0$}
\psfrag{win}{\scriptsize $\omega=\pm \infty$}
\psfrag{xin}{\scriptsize  $0$}
\psfrag{w0p}[r][r]{\scriptsize$\omega=0^{+}$}
\psfrag{w0m}{\scriptsize $\omega=0^{-}$}
\psfrag{wip}{\scriptsize $\omega=+\infty$}
\psfrag{wim}{\scriptsize $\omega=-\infty$}
\subfigure[K=-1]{\label{SubFig:Integral_Stable_NMP_Ny_Ex1}
\psfrag{w0}{\scriptsize$ \omega=0$}
\psfrag{wc}[c][c]{\scriptsize $\omega=1.62$}
\psfrag{win}{\scriptsize $\omega=\infty$}
\psfrag{x0}{\scriptsize $Re\{L(s)\}=-1.5$}
\psfrag{xc}[c][c]{\scriptsize -0.5}
\psfrag{xin}{\scriptsize $0$}
\psfrag{inf}{\scriptsize $\infty$}
\psfrag{R}{\scriptsize $\mathcal{R}_{(-\infty,-1)}$}
\includegraphics[scale=.5]{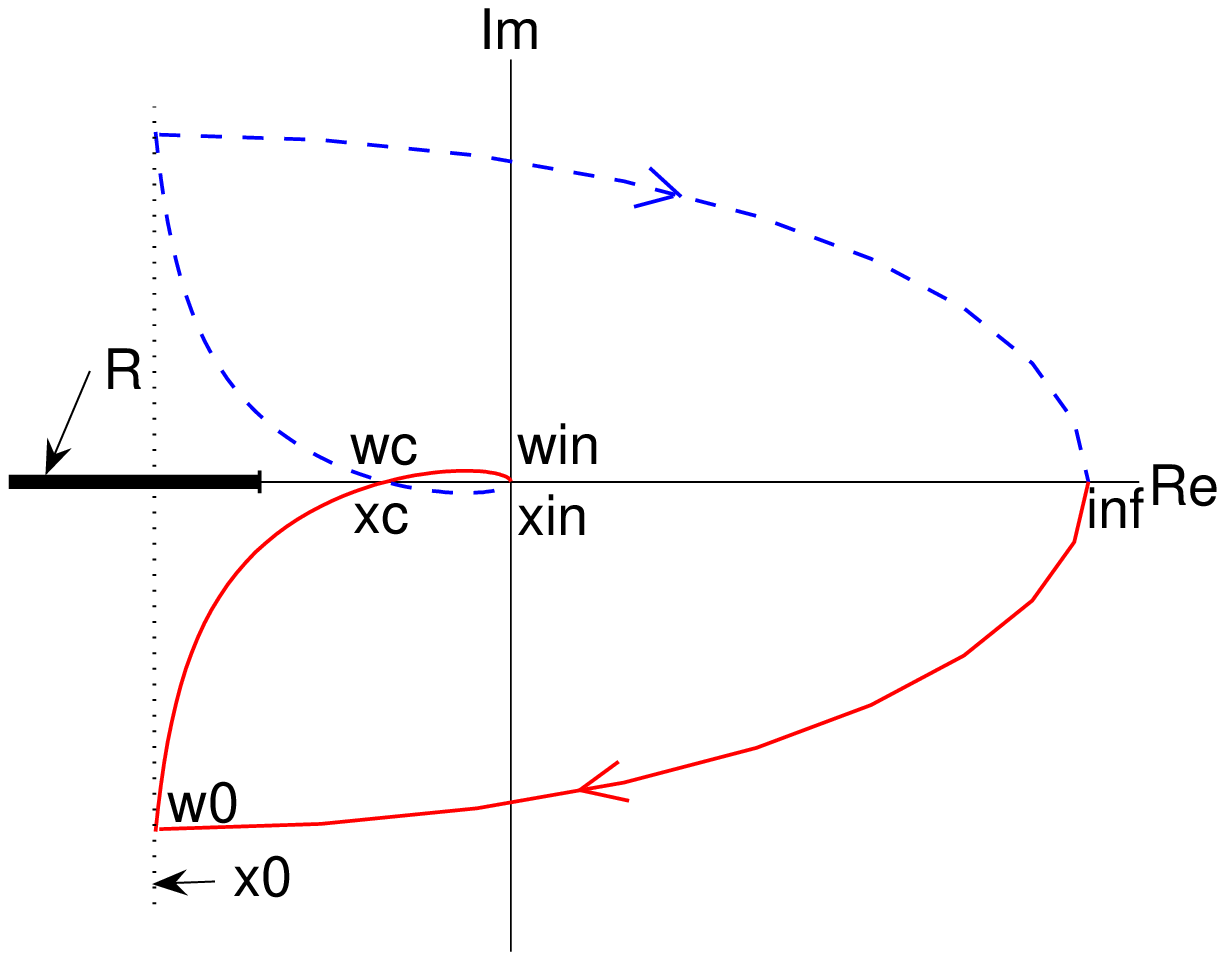}}
\hspace{1em}
\subfigure[K=-1]{\label{SubFig:Integral_Stable_NMP_NC_Ex1}
\psfrag{R}{\scriptsize  $\mathcal{R}_{(-180^\circ,r>0dB)}$}
\psfrag{wc}{\scriptsize $\omega=1.16$}
\psfrag{-40}{\scriptsize $-\infty$}
\includegraphics[scale=.5]{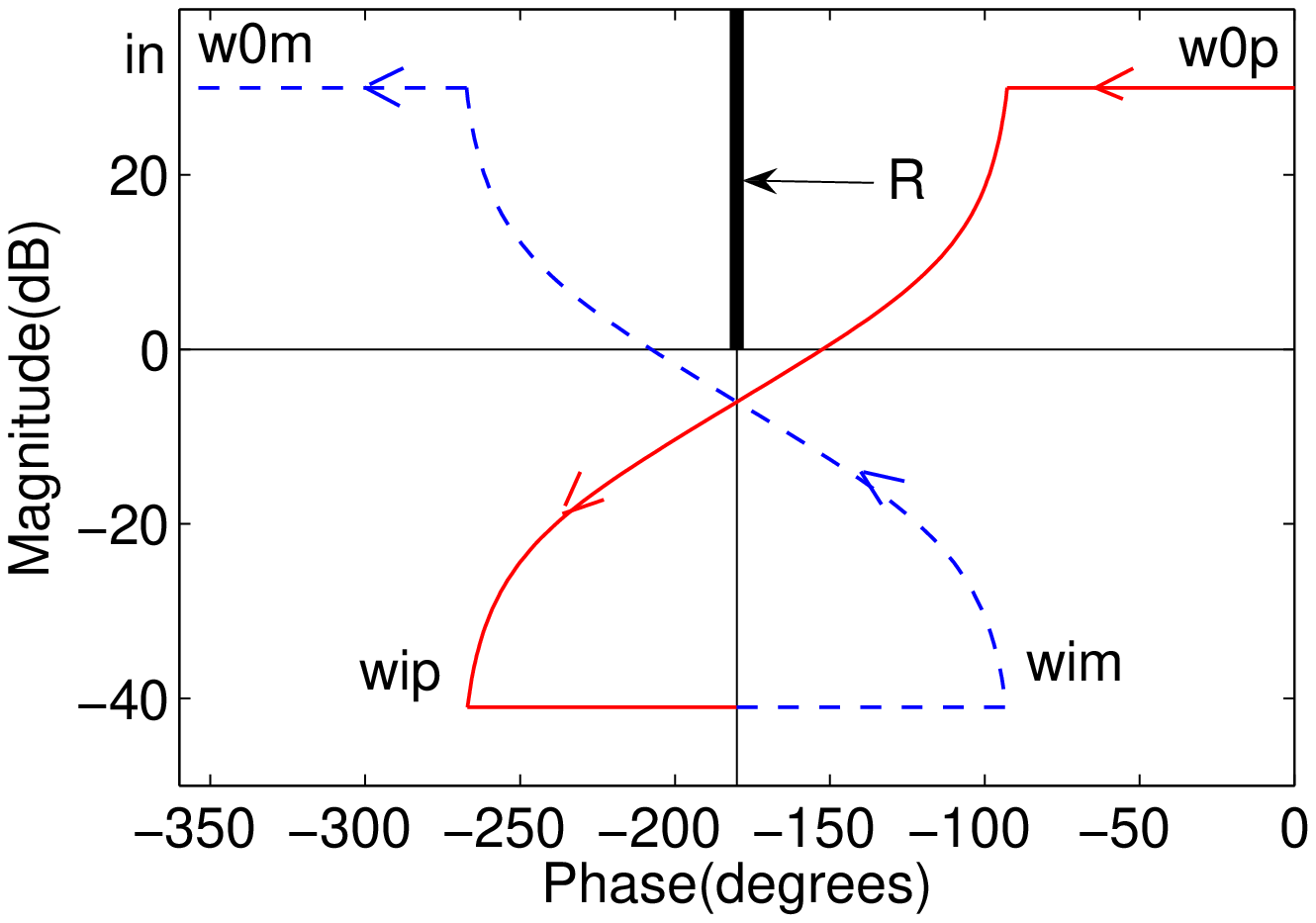}}
\\
\subfigure[K=-5]{\label{SubFig:Integral_Stable_NMP_Ny_Ex2}
\psfrag{w0}{\scriptsize$ \omega=0$}
\psfrag{wc}[c][c]{\scriptsize $\omega=1.42$}
\psfrag{win}{\scriptsize $\omega=\infty$}
\psfrag{x0}{\scriptsize $Re\{L(s)\}=-7.5$}
\psfrag{xc}[c][c]{\scriptsize -2.5}
\psfrag{xin}{\scriptsize $0$}
\psfrag{inf}{\scriptsize $\infty$}
\psfrag{R}{$\mathcal{R}_{(-\infty,-1)}$}
\includegraphics[scale=.5]{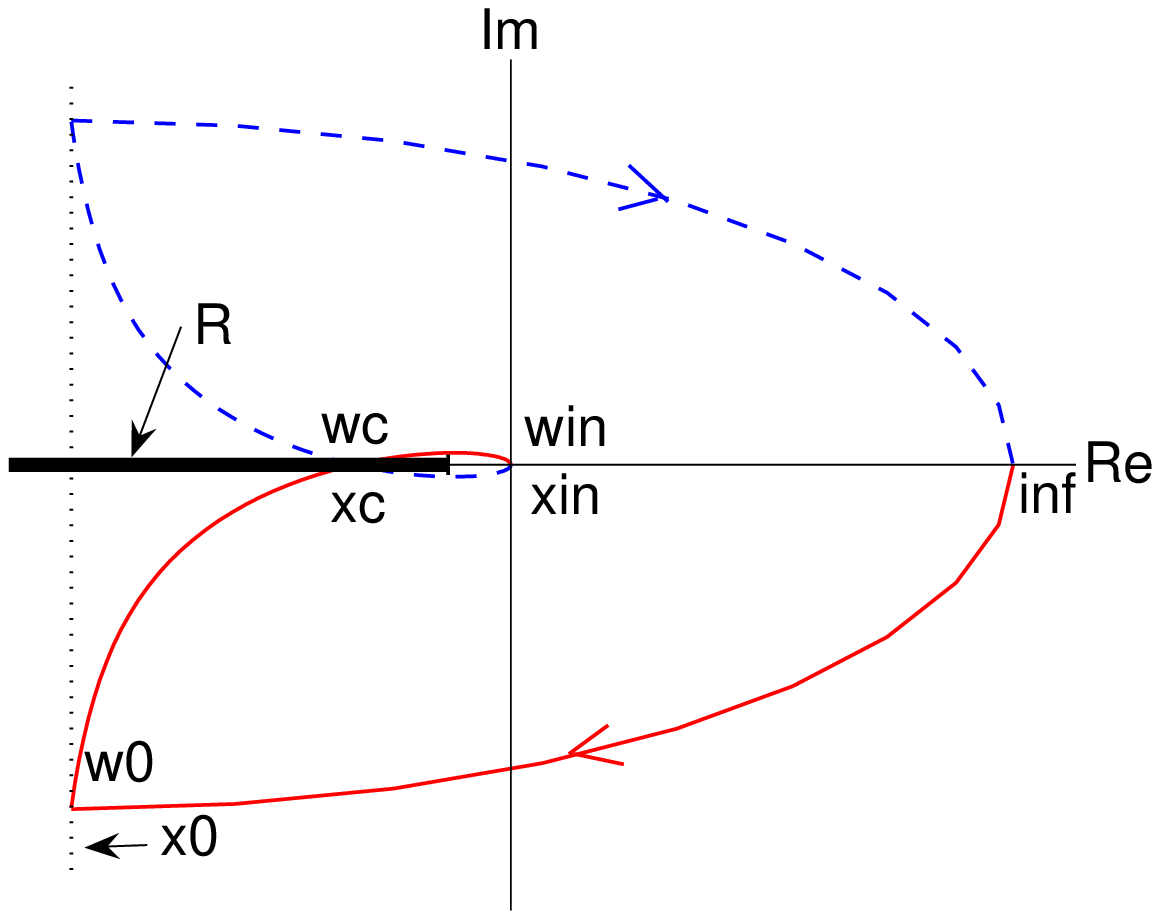}}
\hspace{1em}
\subfigure[K=-5]{\label{SubFig:Intehral_Stable_NMP_NC_Ex2}
\psfrag{R}{\scriptsize  $\mathcal{R}_{(-180^\circ,r>0dB)}$}
\psfrag{wc}{\scriptsize $\omega=1.16$}
\psfrag{-40}{\scriptsize $-\infty$}
\psfrag{60}[r][r]{}
\includegraphics[scale=.5]{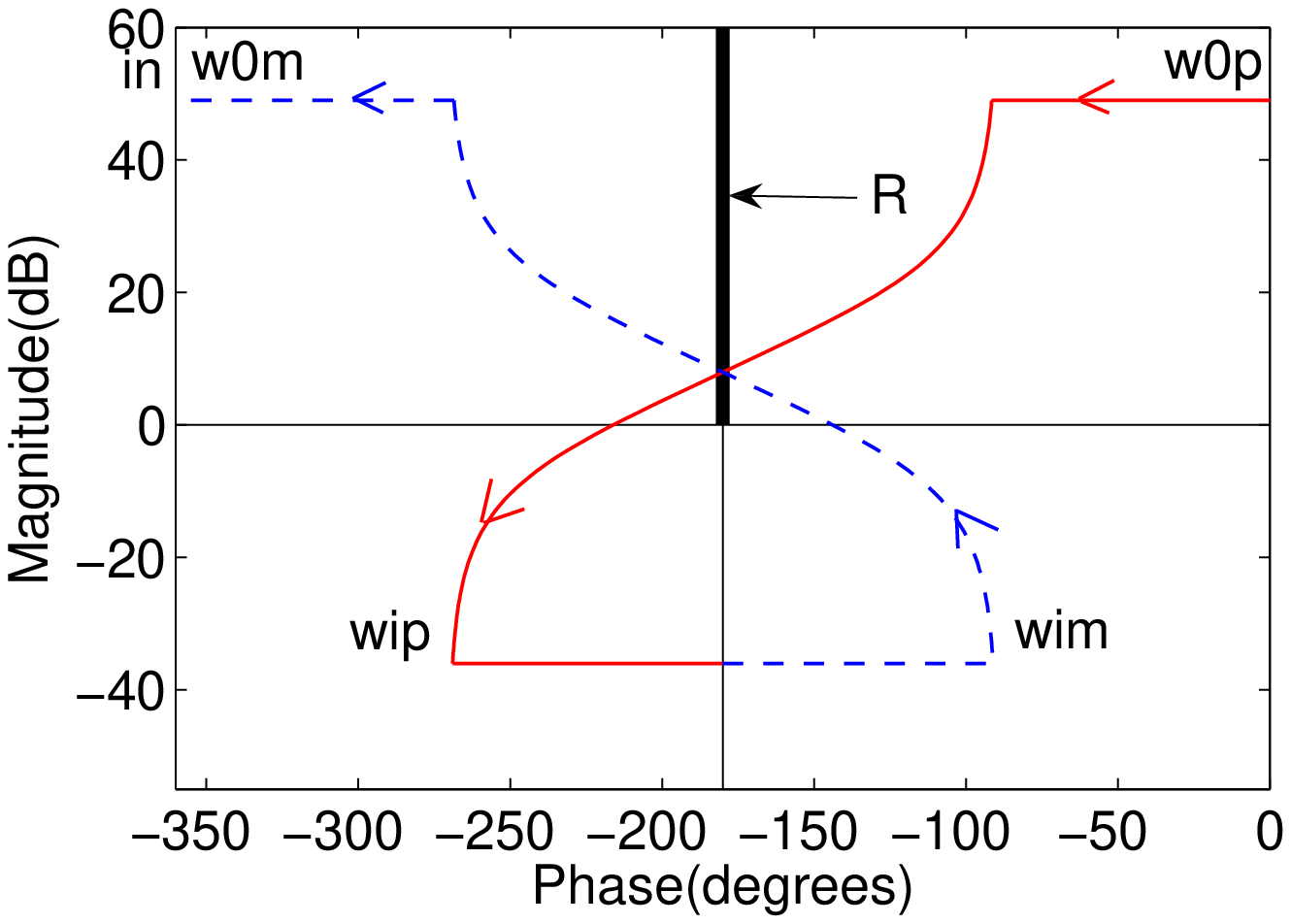}}
\caption{Nyquist plot of Example \ref{Integral_NMP_Stable_Ny_Ex}.}
     \label{Fig:Integral_Stable_NMP_Ny_Ex}
\end{figure}
\begin{example}
\label{Integral_NMP_Stable_Ny_Ex} In this example, the stability conditions is investigated for a system with an integrator and NMP zero. Consider the feedback system \ref{Fig:1DOF_standard_Feedback_structure} with an unstable and non-minimum phase loop-function as given by:
\begin{equation} \label{Eq:Integral_Stable_NMP_Ny_Ex}
L(s)=\frac{K(\frac{s}{2}-1)}{s(\frac{s}{1}+1)}\end{equation}
Using the crossing concept, determine the system stability for $K=-1$ and $K=-5$.
\end{example}
\begin{solution}
The system does not have nay pole inside the standard Nyquist contour, then $N_p=0$. According to the Nyquist diagram of the $L(s)$ as shown Fig. \ref{SubFig:Integral_Stable_NMP_Ny_Ex1}, the number of crossings of $R_{(-\infty,-1)}$ is $0$ for $K=-1$. It means that the Nyquist diagram does not encircle the critical point $(-1,0)$. Using (\ref{Nyquist_Criterion_Eq}), $N_z=N+N_p=0+0=0$ which implies the feedback system is stable for $K=1$. There is also no crossing as shown by Fig. \ref{SubFig:Integral_Stable_NMP_NC_Ex1}. Then, $N \neq -N_p=0$ and the feedback system is stable for $K=-1$.

Subsequently, Fig. \ref{SubFig:Integral_Stable_NMP_Ny_Ex2} shows the Nyquist diagram of $L(s)$ for $K=-5$. The number of crossings of $R_{(-\infty,-1)}$ is $N=+2$, implying that the Nyquist diagram encircles the critical point $(-1,0)$ once in the clockwise direction. By using (\ref{Nyquist_Criterion_Eq}), $N_z=N+N_p=2+0=2$, saying that the feedback system has two unstable poles and is unstable for $K=5$. For $K=-5$, the sum of crossing is $N=2$ as shown by Fig. \ref{SubFig:Intehral_Stable_NMP_NC_Ex2}. Thus, $N \neq -N_p$ and the feedback system is unstable.
\end{solution}

\section{Conclusions}
\label{sec:Conclusions}
In this paper stability analysis techniques based on Nyquist and Nichols charts have been reviewed. The relationship between theses two methodologies has fully been described through several numerical examples. This tutorial provides useful insights into the loop-shaping based control systems design such as Quantitative Feedback Theory.

\section{Appendix}

\subsection{Harry Nyquist (1889-1976)}
Harry Nyquist was born in the Stora Kil parish of Nilsby, Värmland, Sweden in 1889. He emigrated to the USA in 1907. He entered the University of North Dakota in 1912 and received B.S. and M.S. degrees in electrical engineering in 1914 and 1915, respectively. He received a Ph.D. in physics at Yale University in 1917. He worked at AT\&T's Department of Development and Research from 1917 to 1934, and continued when it became Bell Telephone Laboratories that year, until his retirement in 1954. As an engineer at Bell Laboratories, Nyquist contributed to many communications problems. In 1932, he published a classic paper on stability of feedback amplifiers, and his Nyquist stability criterion has been the reference in all textbooks of feedback control theory. Nyquist lived in Pharr, Texas after his retirement, and passed away on April 4, 1976.
\begin{figure}
\centering
\includegraphics[scale=1]{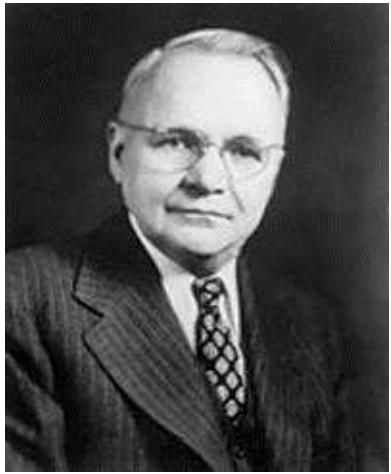}
\caption{Harry Nyquist (1889–1976)}
     \label{Pic-Harry Nyquist}
\end{figure}

\subsection{Nathaniel B. Nichols (1914-1997)\cite{Obituary-NickNichols}}
Nathaniel B. Nichols was born in Michigan in 1914 and earned a BS degree from Central Michigan University in 1936. He went from there to the University of Michigan, earning a MS degree in physics in 1937. He received two Doctor of Philosophy honorary degrees from Case Western Reserve University and from Central Michigan University. In 1942, John Ziegler and Nichols published a paper in the ASME Transactions (Vol. 64, Pg. 759) describing a set of parameters (later known as the Ziegler-Nichols tuning parameters) which were rough approximations to optimal settings for open loop transfer functions to ensure prescribed behavior in closed loop PID realizations. Nichols and his colleagues W.P. Manger and E.H. Krohn published their idea of reading closed loop gain and phase directly from a plot of open loop logrithmic gain and phase which is known as the Nichols chart. The idea was published entitled ``General Design Principles for Servomechanisms" in Chapter 4 of the now famous Volume 25, ``Theory of Servomechanisms" \cite{Nichols-Theory-of-Servomechanisms}. He was President of the IEEE Control Systems Society in 1968 and the American Automatic Control Council in 1974 and in 1975 when the IFAC Congress was first held in the US. After so many technical contributions and services to control engineering community, Nichols was retired from the Control Analysis Department of Aerospace Corporation in El Segundo, California in 1974, and passed away on 17 April 1997.
\begin{figure}
\centering
\includegraphics[scale=1]{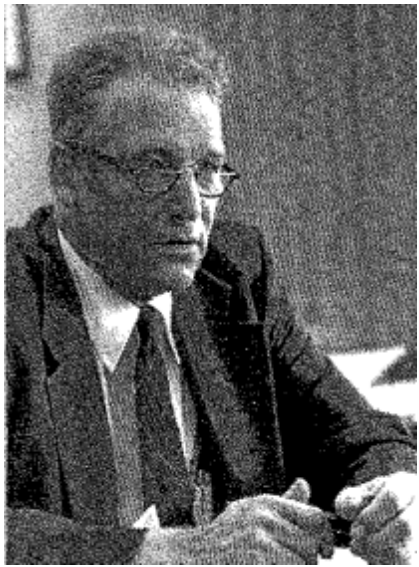}
\caption{Nathaniel B. Nichols (1914-1997)}
     \label{Pic-Nichols}
\end{figure}



\end{document}